\documentclass[11pt,a4paper]{article}
\usepackage{amsmath, amsfonts,amsthm,amssymb,amsbsy,upref,color,graphicx,amscd,hyperref,makeidx,enumerate}
\usepackage{tikz}
\usepackage[active]{srcltx}
\usepackage[latin1]{inputenc}

\def\g{\gamma}

\def\eps{{\varepsilon}}

\def\N{\mathbb{N}}
\def\R{\mathbb{R}}
\def\A{\mathcal{A}}
\def\D{\mathcal{D}}
\def\E{\mathcal{E}}

\def\H{\mathcal{H}}
\def\LL{\mathcal{L}}

\def\V{\mathcal{V}}
\def\<{\langle}
\def\>{\rangle}

\newcommand{\be}{\begin{equation}}
\newcommand{\ee}{\end{equation}}
\newcommand{\bib}[4]{\bibitem{#1}{\sc#2: }{\it#3. }{#4.}}

\numberwithin{equation}{section}
\theoremstyle{plain}
\newtheorem{teo}{Theorem}[section]
\newtheorem{lemma}[teo]{Lemma}
\newtheorem{cor}[teo]{Corollary}
\newtheorem{prop}[teo]{Proposition}

\theoremstyle{remark}
\newtheorem{oss}[teo]{Remark}
\newtheorem{exam}[teo]{Example}

\newenvironment{ack}{{\bf Acknowledgements.}}

\title{Shape Optimization Problems for Metric Graphs}

\author{Giuseppe Buttazzo, Berardo Ruffini and Bozhidar Velichkov}

\begin{document}


\maketitle

\begin{abstract}
We consider the shape optimization problem
$$\min\big\{\E(\Gamma)\ :\ \Gamma\in\A,\ \H^1(\Gamma)=l\ \big\},$$
where $\H^1$ is the one-dimensional Hausdorff measure and $\A$ is an admissible class of one-dimensional sets connecting some prescribed set of points $\D=\{D_1,\dots,D_k\}\subset\R^d$. The cost functional $\E(\Gamma)$ is the Dirichlet energy of $\Gamma$ defined through the Sobolev functions on $\Gamma$ vanishing on the points $D_i$. We analyze the existence of a solution in both the families of connected sets and of metric graphs. At the end, several explicit examples are discussed.
\end{abstract}



\section{Introduction}\label{sintro}

In the present paper we consider the problem of finding optimal graphs in a given admissible class consisting of all connected graphs of prescribed total length and containing a prescribed set of points. The minimization criterion we consider along all the paper is the Dirichlet energy, though in the last section we discuss the possibility of extending our results to other criteria, like the first Dirichlet eigenvalue or similar spectral functionals.

A graph $C$ in $\R^d$ is simply a closed connected subset of $\R^d$ with finite 1-dimensional Hausdorff measure $\H^1(C)$. Since such sets are rectifiable (see for instance \cite[Section 4.4]{at}) we can define all the variational tools that are usually defined in the Euclidean setting:

\begin{itemize}
\item Dirichlet integral $\int_C\frac12 |u'|^2\,d\H^1$;

\item Sobolev spaces
\begin{align}
&H^1(C)=\Big\{u\in L^2(C)\ :\ \int_C|u'|^2\,d\H^1<+\infty\Big\},\nonumber\\
&H^1_0(C;\D)=\Big\{u\in H^1(C)\ :\ u=0\hbox{ on }\D\Big\};\nonumber
\end{align}

\item Energy
$$\E(C;\D)=\inf\Big\{\int_C\Big(\frac12|u'|^2-u\Big)\,d\H^1\ :\ u\in H^1_0(C,\D)\Big\}.$$
\end{itemize}

In particular, for a fixed set $\D$ consisting of $N$ points, $\D=\{D_1,\dots,D_N\}$, we consider the shape optimization problem
\be\label{shopt1}
\min\big\{\E(C;\D)\ :\ \H^1(C)=l,\ \D\subset C\big\},
\ee
where the total length $l$ is fixed. Notice that in the problem above the unknown is the graph $C$ and no a priori constraints on its topology are imposed.

In spite of the fact that the optimization problem \eqref{shopt1} looks very natural, we show that in general an optimal graph may not exist (see Example \ref{example1}); this leads us to consider a larger admissible class consisting of the so-called {\it metric graphs}, for which the embedding into $\R^d$ is not required. The precise definition of a metric graph is given in Section \ref{s3}; roughly speaking they are metric spaces induced by combinatorial graphs with weighted edges.

Our main result is an existence theorem for optimal metric graphs, where the cost functional is the extension of the energy functional defined above. In Section \ref{s5} we show some explicit examples of optimal metric graphs. The last section contains some discussions on possible extensions of our result to other similar problems and on some open questions.

For a review on metric graphs and their application to Physics (where they are commonly called {\it quantum graphs}) we refer to \cite{gnu06}, \cite{Ku}.

\section{Sobolev space and Dirichlet Energy of a rectifiable set}\label{s1}

Let $C\subset\R^d$ be a closed connected set of finite length, i.e. $\H^1(C)<\infty$, where $\H^1$ denotes the one-dimensional Hausdorff measure. On the set $C$ we consider the metric
$$d(x,y)=\inf\left\{\int_0^1|\dot\gamma(t)|\,dt\ :\ \gamma:[0,1]\to\R^d\hbox{ Lipschitz, }\g([0,1])\subset C,\ \g(0)=x,\ \g(1)=y\right\},$$
which is finite since, by the First Rectifiability Theorem (see \cite[Theorem 4.4.1]{at}), there is at least one rectifiable curve in $C$ connecting $x$ to $y$. For any function $u:C\to\R$, Lipschitz with respect to the distance $d$ (we also use the term $d$-Lipschitz), we define the norm
$$\|u\|^2_{H^1(C)}=\int_C |u(x)|^2d\H^1(x)+\int_C |u'|(x)^2d\H^1(x),$$
where
$$|u'|(x)=\limsup_{y\to x}\frac{|u(y)-u(x)|}{d(x,y)}.$$
The Sobolev space $H^1(C)$ is the closure of the $d$-Lipschitz functions on $C$ with respect to the norm $\|\cdot\|_{H^1(C)}$. 

\begin{oss}\label{cpt}
The inclusion $H^1(C)\subset C_d(C)$ is compact, where $C_d(C)$ indicates the space of real-valued continuous functions on $C$, with respect to the metric $d$. In fact, for each $x,y\in C$, there is a rectifiable curve $\gamma:[0,d(x,y)]\to C$ connecting $x$ to $y$, which we may assume arc-length parametrized. Thus, for any $u\in H^1(C)$, we have that
\begin{align}
|u(x)-u(y)|&\le\int_0^{d(x,y)}\left|\frac{d}{dt}u(\gamma(t))\right|\,dt\nonumber\\
&\le d(x,y)^{1/2}\left(\int_0^{d(x,y)}\left|\frac{d}{dt}u(\gamma(t))\right|^2dt\right)^{1/2}\nonumber\\
&\le d(x,y)^{1/2}\|u'\|_{L^2(C)},\nonumber
\end{align}
and so, $u$ is $1/2$-H\"older continuous. On the other hand, for any $x\in C$, we have that
$$\int_C u(y)d\H^1(y)\ge \int_C\left( u(x)-d(x,y)^{1/2}\|u'\|_{L^2(C)}\right)\,d\H^1(y)\ge lu(x)-l^{3/2}\|u'\|_{L^2(C)},$$
where $l=\H^1(C)$. Thus, we obtain the $L^\infty$ bound
$$\|u\|_{L^\infty}\le l^{-1/2}\|u\|_{L^2(C)}+l^{1/2}\|u'\|_{L^2(C)}\le(l^{-1/2}+l^{1/2})\|u\|_{H^1(C)}.$$
and so, by the Ascoli-Arzel\'a Theorem, we have that the inclusion is compact.
\end{oss}

\begin{oss}\label{pncr}
By the same argument as in Remark \ref{cpt} above, we have that for any $u\in H^1(C)$, the $(1,2)$-Poincar\'e inequality holds, i.e.
\be\label{pncr1}
\int_C\left|u(x)-\frac1l\int_C u\,d\H^1\right|d\H^1(x)\le l^{3/2}\left(\int_C|u'|^2d\H^1\right)^{1/2}.
\ee
Moreover, if $u\in H^1(C)$ is such that $u(x)=0$ for some point $x\in C$, then we have the Poincar\'e inequality:
\be\label{pncr}
\|u\|_{L^2(C)}\le l^{1/2}\|u\|_{L^{\infty}(C)}\le l\|u'\|_{L^2(C)}.
\ee
\end{oss}

Since $C$ is supposed connected, by the Second Rectifiability Theorem (see \cite[Theorem 4.4.8]{at}) there exists a countable family of injective arc-length parametrized Lipschitz curves $\g_i:[0,l_i]\to C$, $i\in\N$ and an $\H^1$-negligible set $N\subset C$ such that
$$C=N\cup\left(\bigcup_i Im(\g_i)\right),$$
where $Im(\g_i)=\g_i([0,l_i])$. By the chain rule (see Lemma \ref{lchain} below) we have
$$\Big|\frac{d}{dt}u(\g_i(t))\Big|=|u'|(\g_i(t)),\qquad\forall i\in\N$$
and so, we obtain for the norm of $u\in H^1(C)$:
\be\label{norm}
\|u\|^2_{H^1(C)}=\int_C |u(x)|^2d\H^1(x)+\sum_i\int_0^{l_i}\left|\frac{d}{dt}u(\gamma_i(t))\right|^2dt.
\ee
Moreover, we have the inclusion 
\be\label{refl}
H^1(C)\subset \oplus_{i\in\N} H^1([0,l_i]),
\ee
which gives the reflexivity of $H^1(C)$ and the lower semicontinuity of the $H^1(C)$ norm, with respect to the strong convergence in $L^2(C)$.

\begin{lemma}\label{lchain}
Let $u\in H^1(C)$ and let $\g:[0,l]\to\R^d$ be an arc-length parametrized Lipschitz curve with $\g([0,l])\subset C$. Then we have
\be\label{chain}
\left|\frac{d}{dt}u(\gamma(t))\right|=|u'|(\g(t)),
\qquad\hbox{for $\LL^1$-a.e. $t\in[0,l]$}.
\ee
\end{lemma}

\begin{proof}
With no loss of generality we may assume that $u:C\to\R$ is a Lipschitz map with Lipschitz constant $Lip(u)$ with respect to the distance $d$ and that the curve $\gamma$ is injective. We prove that the chain rule \eqref{chain} holds in all the points $t\in[0,l]$ which are Lebesgue points for $\left|\frac{d}{dt}u(\g(t))\right|$ and such that the point $\gamma(t)$ has density one, i.e.
\begin{equation}\label{dens}
\lim_{r\to0}\frac{\H^1\big(C\cap B_r(\gamma(t))\big)}{2r}=1,
\end{equation}
(thus almost every points, see for instance \cite[Theorem I.10.2]{ma12}) where $B_r(x)$ indicates the ball of radius $r$ in $\R^d$. Since, $\H^1$-almost all points $x\in C$ have this property, we obtain the conclusion. Without loss of generality, we consider $t=0$. Let us first prove that $|u'|(\g(0))\ge \left|\frac{d}{dt}u(\g(0))\right|$. We have that 
$$|u'|(\gamma(0))\ge \limsup_{t\to0}\frac{|u(\gamma(t))-u(\gamma(0))|}{d(\gamma(t),\gamma(0))}=\left|\frac{d}{dt}u(\g(0))\right|,$$
since $\gamma$ is arc-length parametrized. On the other hand, we have
\begin{align}
|u'|(x)&=\limsup_{y\to x}\frac{|u(y)-u(x)|}{d(y,x)}\nonumber\\
&=\lim_{n\to\infty}\frac{|u(y_n)-u(x)|}{d(y_n,x)}\nonumber\\
&=\lim_{n\to\infty}\frac{|u(\gamma_n(r_n))-u(\gamma_n(0))|}{r_n}\nonumber\\
&\le\lim_{n\to\infty}\frac{1}{r_n}\int_0^{r_n}\left|\frac{d}{dt}u(\gamma_n(t))\right|\,dt\label{chain3}
\end{align}
where $y_n\in C$ is a sequence of points which realizes the $\limsup$ and $\gamma_n:[0,r_n]\to\R^d$ is a geodesic in $C$ connecting $x$ to $y_n$. Let $S_n=\{t:\g_n(t)=\g(t)\}\subset [0,r_n]$, then, we have
\begin{align}
\int_{0}^{r_n}\left|\frac{d}{dt}u(\g_n(t))\right|^2dt
&\le\int_{S_n}\left|\frac{d}{dt}u(\g(t))\right|^2dt+Lip(u)\,(r_n-|S_n|)\nonumber\\
&\le\int_{0}^{r_n}\left|\frac{d}{dt}u(\g(t))\right|^2dt+Lip(u)\,(\H^1(B_{r_n}(\g(0))\cap C)-2r_n),\label{chain1}
\end{align}
and so, since $\g(0)$ is of density $1$, we conclude applying this estimate to \eqref{chain3}. 
\end{proof}

Given a set of points $\D=\{D_1,\dots,D_k\}\subset\R^d$ we define the admissible class $\A(\D;l)$ as the family of all closed connected sets $C$ containing $\D$ and of length $\H^1(C)=l$. For any $C\in \A(\D;l)$ we consider the space of Sobolev functions which satisfy a Dirichlet condition at the points $D_i$:
$$H^1_0(C;\D)=\{u\in H^1(C): u(D_j)=0,\, j=1\dots,k\},$$
which is well-defined by Remark \ref{cpt}. For the points $D_i$ we use the term \emph{Dirichlet points}. The \emph{Dirichlet Energy} of the set $C$ with respect to $D_1,\dots,D_k$ is defined as
\be\label{energy1}
\E(C;\D)=\min\left\{J(u)\ :\ u\in H^1_0(C;\D)\right\},
\ee
where 
\be\label{energy2}
J(u)=\frac12\int_C|u'|(x)^2\,d\H^1(x)-\int_Cu(x)\,d\H^1(x).
\ee

\begin{oss}
For any $C\in\A(\D;l)$ there exists a unique minimizer of the functional $J:H^1_0(C;\D)\to\R$. In fact, by Remark \ref{cpt} we have that a minimizing sequence is bounded in $H^1$ and compact in $L^2$. The conclusion follows by the semicontinuity of the $L^2$ norm of the gradient, with respect to the strong $L^2$ convergence, which is an easy consequence of equation \eqref{norm}. The uniqueness follows by the strict convexity of the $L^2$ norm and the sub-additivity of the gradient $|u'|$. We call the minimizer of $J$ the \emph{energy function} of $C$ with Dirichlet conditions in $D_1,\dots,D_k$.
\end{oss}

\begin{oss}\label{coarea}
Let $u\in H^1(C)$ and $v:C\to\R$ be a positive Borel function. Applying the chain rule, as in \eqref{norm}, and the one dimensional co-area formula (see for instance \cite[Theorem 3.40]{amfupa}), we obtain a co-area formula for the functions $u\in H^1(C)$:
\begin{align}
\int_C v(x)|u'|(x)\,d\H^1(x)&=\sum_i\int_0^{l_i}\left|\frac{d}{dt}u(\gamma_i(t))\right|v(\gamma_i(t))\,dt\nonumber\\
&=\sum_i\int_0^{+\infty}\!\!\Big(\!\!\!\sum_{u\circ\g_i(t)=\tau}\!\!\!v\circ\gamma_i(t)\Big)\,d\tau\label{coa}\\
&=\int_0^{+\infty}\!\!\Big(\sum_{u(x)=\tau}\!\!\!v(x)\Big)\,d\tau.\nonumber
\end{align}
\end{oss}

\subsection{Optimization problem for the Dirichlet Energy on the class of connected sets}\label{s2}

We study the following shape optimization problem:
\be\label{minenergy1}
\min\left\{\E(C;\D)\,:\,C\in\A(\D;l)\right\},
\ee
where $\D=\{D_1,...,D_k\}$ is a given set of points in $\R^d$ and $l$ is a prescribed length. 

\begin{oss}\label{ps}
When $k=1$ problem $\eqref{minenergy1}$ reads as
\be\label{minenergy1p}
\E=\min\left\{\E(C;D)\ :\ \H^1(C)=l,\ D\in C\right\},
\ee
where $D\in\R^d$ and $l>0$. In this case the solution is a line of length $l$ starting from $D$ (see Figure \ref{f1ex0}). A proof of this fact, in a slightly different context, can be found in \cite{fr} and we report it here for the sake of completeness.

\begin{figure}[h]
\centerline{\includegraphics[width=8.0cm]{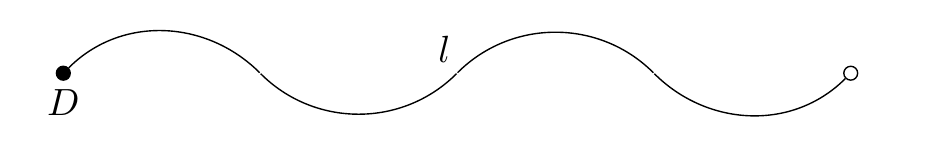}}
\caption{The optimal graph with only one Dirichlet point.}
\label{f1ex0}
\end{figure}

Let $C\in\A(D;l)$ be a generic connected set and let $w\in H^1_0(C;D)$ be its energy function, i.e. the minimizer of $J$ on $C$. Let $v:[0,l]\to\R$ be such that $\mu_w(\tau)=\mu_v(\tau)$, where $\mu_w$ and $\mu_v$ are the distribution function of $w$ and $v$ respectively, defined by
$$\mu_w(\tau)=\H^1(w\le \tau)=\sum_i\H^1(w_i\le \tau),\qquad \mu_v(\tau)=\H^1(v\le \tau).$$
It is easy to see that, by the Cavalieri Formula, $\|v\|_{L^p([0,l])}=\|w\|_{L^p(C)}$, for each $p\ge1$. By the co-area formula \eqref{coa}
\be\label{e1r3}
\int_C|w'|^2\,d\H^1
=\int_0^{+\infty}\!\!\!\Big(\sum_{w=\tau}|w'|\Big)\,d\tau
\ge\int_0^{+\infty}\!\!\!\Big(\sum_{w=\tau}\frac{1}{|w'|}\Big)^{-1}\!\!\!d\tau
=\int_0^{+\infty}\frac{d\tau}{\mu'_w(\tau)},
\ee
where we used the Cauchy-Schwartz inequality and the identity
$$\mu_w(t)=\H^1(\{w\le t\})=\int_{w\le t}\frac{|w'|}{|w'|}\,ds=\int_0^t\Big(\sum_{w=s}\frac{1}{|w'|}\Big)\,ds$$
which implies that $\mu_w'(t)=\sum_{w=t}\frac{1}{|w'|}$. The same argument applied to $v$ gives:
\be\label{e2r3}
\int_0^l|v'|^2\,dx=\int_0^{+\infty}\Big(\sum_{v=\tau}|v'|\Big)\,d\tau=\int_0^{+\infty}\frac{d\tau}{\mu_v'(\tau)}.
\ee
Since $\mu_w=\mu_v$, the conclusion follows.
\end{oss}

The following Theorem shows that it is enough to study the problem \eqref{minenergy1} on the class of finite graphs embedded in $\R^d$. Consider the subset $\A_{N}(\D;l)\subset\A(\D;l)$ of those sets $C$ for which there exists a finite family $\gamma_i:[0,l_i]\to\R$, $i=1,\dots,n$ with $n\le N$, of injective rectifiable curves such that $\cup_i\g_i([0,l_i])=C$ and $\g_i((0,l_i))\cap\g_j((0,l_j))=\emptyset$, for each $i\ne j$.

\begin{teo}\label{th1}
Consider the set of distinct points $\D=\{D_1,\dots,D_k\}\subset\R^d$ and $l>0$. We have that
\be\label{inf1}
\inf\big\{\E(C;\D)\ :\ C\in \A(\D;l)\big\}=\inf\big\{\E(C;\D)\ :\ C\in \A_N(\D;l)\big\},
\ee
where $N=2k-1$. Moreover, if $C$ is a solution of the problem \eqref{minenergy1}, then there is also a solution $\widetilde C$ of the same problem such that $\widetilde C\in\A_N(\D;l)$.
\end{teo}

\begin{proof}
Consider a connected set $C\in\A(\D;l)$. We show that there is a set $\widetilde C\in\A_N(\D;l)$ such that $\E(\widetilde C;\D)\le \E(C;\D)$. Let $\eta_1:[0,a_1]\to C$ be a geodesic in $C$ connecting $D_1$ to $D_2$ and let $\eta_2:[0,a]\to C$ be a geodesic connecting $D_3$ to $D_1$. Let $a_2$ be the smallest real number such that $\eta_2(a_2)\in\eta_1([0,a_1])$. Then, consider the geodesic $\eta_3$ connecting $D_4$ to $D_1$ and the smallest real number $a_3$ such that $\eta_3(a_3)\in\eta_1([0,a_1])\cup\eta_2([0,a_2])$. Repeating this operation, we obtain a family of geodesics $\eta_i$, $i=1,\dots,k-1$ which intersect each other in a finite number of points. Each of these geodesics can be decomposed in several parts according to the intersection points with the other geodesics (see Figure \ref{f1th1}).

\begin{figure}[h]
\centerline{\includegraphics[width=10.0cm]{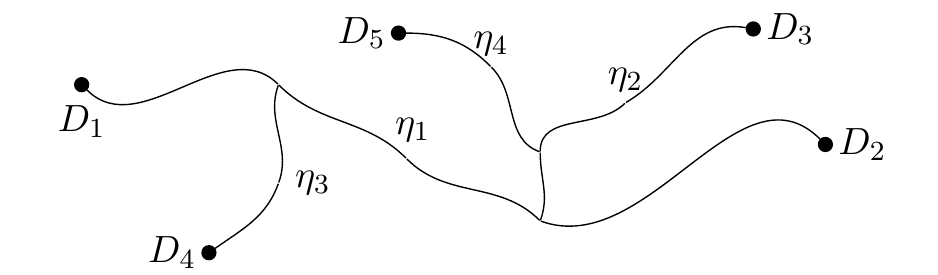}}
\caption{Construction of the set $C'$.}
\label{f1th1}
\end{figure}

So, we can consider a new family of geodesics (still denoted by $\eta_i$), $\eta_i:[0,l_i]\to C$, $i=1,\dots,n$, which does not intersect each other in internal points. Note that, by an induction argument on $k\ge 2$, we have $n\le 2k-3$. Let $C'=\cup_i\eta_i([0,l_i])\subset C$. By the Second Rectifiability  Theorem (see \cite[Theorem 4.4.8]{at}), we have that 
$$C=C'\cup E\cup\Gamma,$$
where $\H^1(E)=0$ and $\Gamma=\left(\bigcup_{j=1}^{+\infty}\g_j\right)$, where $\g_j:[0,l_j]\to C$ for $j\ge 1$ is a family of Lipschitz curves in $C$. Moreover, we can suppose that $\H^1(\Gamma\cap C')=0$. In fact, if $\H^1(Im(\g_j)\cap C')\neq0$ for some $j\in\N$, we consider the restriction of $\gamma_j$ to (the closure of) each connected component of $\g_j^{-1}(\R^d\setminus C')$.

Let $w\in H^1_0(C;\D)$ be the energy function on $C$ and let $v:[0,\H^1(\Gamma)]\to\R$ be a monotone increasing function such that $|\{v\le\tau\}|=\H^1(\{w\le\tau\}\cap\Gamma)$. Reasoning as in Remark \ref{ps}, we have that 
\be\label{e1t1}
\frac12\int_0^{\H^1(\Gamma)}|v'|^2\,dx-\int_0^{\H^1(\Gamma)}v\,dx
\le\frac12\int_{\Gamma}|w'|^2\,d\H^1-\int_{\Gamma}w\,d\H^1.
\ee

Let $\sigma:[0,\H^1(\Gamma)]\to\R^d$ be an injective arc-length parametrized curve such that $Im(\sigma)\cap C'=\sigma(0)=x'$, where $x'\in C'$ is the point where $w_{|C'}$ achieves its maximum. Let $\widetilde C=C'\cup Im(\sigma)$. Notice that $\widetilde C$ connects the points $D_1,\dots,D_k$ and has length $\H^1(\widetilde C)=\H^1(C')+\H^1(Im(\sigma))=\H^1(C')+\H^1(\Gamma)=l$. Moreover, we have 
\be\label{endecr}
\E(\widetilde C;\D)\le J(\widetilde w)\le J(w)=\E(C;\D),
\ee
where $\widetilde w$ is defined by
\be
\widetilde w(x)=
\begin{cases}
w(x),&\hbox{if }x\in C',\\
v(t)+w(x')-v(0),&\hbox{if }x=\sigma(t).
\end{cases}
\ee
We have then \eqref{endecr}, i.e. the energy decreases. We conclude by noticing that the point $x'$ where we attach $\sigma$ to $C'$ may be an internal point for $\eta_i$, i.e. a point such that $\eta^{-1}_i(x')\in(0,l_i)$. Thus, the set $\widetilde C$ is composed of at most $2k-1$ injective arc-length parametrized curves which does not intersect in internal points, i.e. $\widetilde C\in \A_{2k-1}(\D;l)$. 
\end{proof}

\begin{oss}\label{optcondr}
Theorem \ref{th1} above provides a nice class of admissible sets, where to search a minimizer of the energy functional $\E$. Indeed, according to its proof, we may limit ourselves to consider only graphs $C$ such that:
\begin{enumerate}
\item $C$ is a tree, i.e. it does not contain any closed loop;
\item the Dirichlet points $D_i$ are vertices of degree one (endpoints) for $C$;
\item there are at most $k-1$ other vertices; if a vertex has degree three or more, we call it Kirchhoff point;
\item there is at most one vertex of degree one for $C$ which is not a Dirichlet point. In this vertex the energy function $w$ satisfies Neumann boundary condition $w'=0$ and so we call it Neumann point.
\end{enumerate}
The previous properties are also necessary conditions for the optimality of the graph $C$ (see Proposition \ref{oc} for more details).  
\end{oss}

As we show in Example \ref{example1}, the problem \eqref{minenergy1} may not have a solution in the class of connected sets. It is worth noticing that the lack of existence only occurs for particular configurations of the Dirichlet points $D_i$ and not because of some degeneracy of the cost functional $\E$. In fact, we are able to produce other examples in which an optimal graph exists (see Section \ref{s5}).

\section{Sobolev space and Dirichlet Energy of a metric graph}\label{s3}

Let $V=\{V_1,\dots,V_N\}$ be a finite set and let $E\subset\big\{e_{ij}=\{V_i ,V_j \}\big\}$ be a set of pairs of elements of $V$. We define combinatorial graph (or just graph) a pair $\Gamma=(V,E)$. We say the set $V=V(\Gamma)$ is the set of vertices of $\Gamma$ and the set $E=E(\Gamma)$ is the set of edges. We denote with $|E|$ and $|V|$ the cardinalities of $E$ and $V$ and with $\deg(V_i)$ the degree of the vertex $V_i$, i.e. the number of edges incident to $V_i$.

A \emph{path} in the graph $\Gamma$ is a sequence $V_{\alpha_0},\dots,V_{\alpha_n}\in V$ such that for each $k=0,\dots,n-1$, we have that $\{V_{\alpha_k} ,V_{\alpha_{k+1}} \}\in E$. With this notation, we say that the path connects $V_{i_0}$ to $V_{i_\alpha}$. The path is said to be \emph{simple} if there are no repeated vertices in $V_{\alpha_0},\dots,V_{\alpha_n}$. We say that the graph $\Gamma=(V,E)$ is connected, if for each pair of vertices $V_i,V_j\in V$ there is a path connecting them. We say that the connected graph $\Gamma$ is a tree, if after removing any edge, the graph becomes not connected.

If we associate a non-negative length (or weight) to each edge, i.e. a map $l:E(\Gamma)\rightarrow [0,+\infty)$, then we say that the couple $(\Gamma,l)$ determines a metric graph of length
$$l(\Gamma):=\sum_{i<j}l(e_{ij}).$$

A function $u:\Gamma\rightarrow\R^n$ on the metric graph $\Gamma$ is a collection of functions $u_{ij}:[0,l_{ij}]\rightarrow\R$, for $1\le i\neq j\le N$, such that:
\begin{enumerate}
\item $u_{ji}(x)=u_{ij}(l_{ij}-x)$, for each $1\le i\neq j\le N$,
\item $u_{ij}(0)=u_{ik}(0)$, for all $\{i,j,k\}\subset\{1,\dots,N\}$,
\end{enumerate}
where we used the notation $l_{ij}=l(e_{ij})$. A function $u:\Gamma\rightarrow\R$ is said continuous ($u\in C(\Gamma)$), if $u_{ij}\in C([0,l_{ij}])$, for all $i,j\in\{1,\dots,n\}$. We call $L^p(\Gamma)$ the space of $p$-summable functions ($p\in[1,+\infty)$), i.e. the functions $u=(u_{ij})_{ij}$ such that 
$$\|u\|_{L^p(\Gamma)}^p:=\frac{1}{2}\sum_{i,j}\|u_{ij}\|_{L^p(0,l_{ij})}^p <+\infty,$$
where $\|\cdot\|_{L^p(a,b)}$ denotes the usual $L^p$ norm on the interval $[a,b]$. As usual, the space $L^2(\Gamma)$ has a Hilbert structure endowed by the scalar product:
$$\<u,v\>_{L^2(\Gamma)}:=\frac12\sum_{i,j}\<u_{ij},v_{ij}\>_{L^2(0,l_{ij})}.$$


We define the Sobolev space $H^1(\Gamma)$ as:
\be
H^1(\Gamma)=\big\{u\in C(\Gamma): u_{ij}\in H^1([0,l_{ij}]),\ \forall i,j\in\{1,\dots,n\}\big\},
\ee
which is a Hilbert space with the norm
\be
\|u\|_{H^1(\Gamma)}^2=\frac12\sum_{i,j}\|u_{ij}\|_{H^1([0,l_{ij}])}^2=\frac12\sum_{i,j}\left(\int_0^{l_{ij}}|u_{ij}|^2dx+\int_0^{l_{ij}}|u'_{ij}|^2dx\right).
\ee

\begin{oss}
Note that for $u\in H^1(\Gamma)$ the family of derivatives $\big(u_{ij}'\big)_{1\le i\neq j\le N}$ is not a function on $\Gamma$, since $u_{ij}'(x)=\frac{\partial}{\partial x}u_{ji}(l_{ij}-x)=-u_{ji}'(l_{ij}-x)$. Thus, we work with the function
$|u'|=\big(|u_{ij}'|\big)_{1\le i\ne j\le N}\in L^2(\Gamma)$.
\end{oss}

\begin{oss}\label{cptgam}
The inclusions $H^1(\Gamma)\subset C(\Gamma)$ and $H^1(\Gamma)\subset L^2(\Gamma)$ are compact, since the corresponding inclusions, for each of the intervals $[0,l_{ij}]$, are compact. By the same argument, the $H^1$ norm is lower semicontinuous with respect to the strong $L^2$ convergence of the functions in $H^1(\Gamma)$.
\end{oss}

For any subset $W=\{W_1,\dots,W_k\}$ of the set of vertices $V(\Gamma)=\left\{V_1,\dots,V_N\right\}$, we introduce the Sobolev space with \emph{Dirichlet boundary conditions} on $W$:
\be
H^1_0(\Gamma;W)=\big\{u\in H^1(\Gamma)\ :\ u(W_1)=\dots=u(W_k)=0\big\}.
\ee

\begin{oss}
Arguing as in Remark \ref{cpt} we have that for each $u\in H^1_0(\Gamma;W)$ and, more generally, for each $u\in H^1(\Gamma)$ such that $u(V_\alpha)=0$ for some $\alpha=1,\dots,N$, the Poincar\'e inequality
\be\label{pncr3}
\|u\|_{L^2(\Gamma)}\le l^{1/2}\|u\|_{L^\infty}\le l\|u'\|_{L^2(\Gamma)},
\ee
holds, where
$$\|u'\|^2_{L^2(\Gamma)}:=\int_{\Gamma}|u'|^2\,dx:=\sum_{i,j}\int_0^{l_{ij}}|u'_{ij}|^2\,dx.$$
\end{oss}



On the metric graph $\Gamma$, we consider the Dirichlet Energy with respect to $W$:
\be\label{energy}
\E(\Gamma;W)=\inf\big\{J(u)\ :\ u\in H^1_0(\Gamma;W)\big\},
\ee
where the functional $J:H^1_0(\Gamma;W)\to\R$ is defined by
\be
J(u)=\frac12\int_{\Gamma}|u'|^2dx-\int_{\Gamma}u\,dx.
\ee

\begin{lemma}\label{enfun}
Given a metric graph $\Gamma$ of length $l$ and Dirichlet points $\{W_1,\dots,W_k\}\subset V(\Gamma)=\{V_1,\dots,V_N\}$, there is a unique function $w=(w_{ij})_{1\le i\ne j\le N}\in H^1_0(\Gamma;W)$ which minimizes the functional $J$. Moreover, we have
\begin{enumerate}[(i)]
\item for each $1\le i\ne j\le N$ and each $t\in(0,l_{ij})$, $-w_{ij}''=1$;
\item at every vertex $V_i\in V(\Gamma)$, which is not a Dirichlet point, $w$ satisfies the Kirchhoff's law: 
$$\sum_{j}w_{ij}'(0)=0,$$
where the sum is over all $j$ for which the edge $e_{ij}$ exists;
\end{enumerate} 
Furthermore, the conditions $(i)$ and $(ii)$ uniquely determine $w$.
\end{lemma}

\begin{proof}
The existence is a consequence of Remark \ref{cptgam} and the uniqueness is due to the strict convexity of the $L^2$ norm.
For any $\varphi\in H^1_0(\Gamma;W)$, we have that $0$ is a critical point for the function
$$\eps\mapsto\frac{1}{2}\int_\Gamma|(w+\eps\varphi)'|^2dx-\int_\Gamma (w+\eps\varphi)\,dx.$$
Since $\varphi$ is arbitrary, we obtain the first claim. The Kirchhoff's law at the vertex $V_i$ follows by choosing $\varphi$ supported in a ``small neighborhood'' of $V_i$. The last claim is due to the fact that if $u\in H^1_0(\Gamma;W)$ satisfies $(i)$ and $(ii)$, then it is an extremal for the convex functional $J$ and so, $u=w$.
\end{proof}

\begin{oss}\label{coarea2}
As in Remark \ref{coarea} we have that the co-area formula holds for the functions $u\in H^1(\Gamma)$ and any positive Borel (on each edge) function $v:\Gamma\to\R$:
\begin{align}
\int_\Gamma v(x)|u'|(x)\,dx
&=\sum_{1\le i<j\le N}\int_0^{l_{ij}}|u'_{ij}(x)|\,v(x)\,dx\nonumber\\
&=\sum_{1\le i<j\le N}\int_0^{+\infty}\!\!\Big(\!\!\!\sum_{u_{ij}(x)=\tau}\!\!\!v(x)\Big)\,d\tau\label{coa2}\\
&=\int_0^{+\infty}\!\!\Big(\sum_{u(x)=\tau}\!\!\!v(x)\Big)\,d\tau.\nonumber
\end{align}
\end{oss}

\subsection{Optimization problem for the Dirichlet Energy on the class of metric graphs}\label{s4}

We say that the continuous function $\gamma=(\g_{ij})_{1\le i\neq j\le N}:\Gamma\to\R^d$ is an \emph{immersion} of the metric graph $\Gamma$ into $\R^d$, if for each $1\le i\neq j\le N$ the function $\g_{ij}:[0,l_{ij}]\to\R^d$ is an injective arc-length parametrized curve. We say that $\gamma:\Gamma\to\R^d$ is an \emph{embedding}, if it is an immersion which is also injective, i.e. for any $i\ne j$ and $i'\ne j'$, we have 
\begin{enumerate}
\item $\g_{ij}((0,l_{ij}))\cap\g_{i'j'}([0,l_{i'j'}])=\emptyset$,
\item $\g_{ij}(0)=\g_{i'j'}(0)$, if and only if, $i=i'$.
\end{enumerate}

\begin{oss}\label{iso}
Suppose that $\Gamma$ is a metric graph of length $l$ and that $\gamma:\Gamma\to\R^d$ is an embedding. Then the set $C:=\gamma(\Gamma)$ is rectifiable of length $\H^1(\gamma(\Gamma))=l$ and the spaces $H^1(\Gamma)$ and $H^1(C)$ are isometric as Hilbert spaces, where the isomorphism is given by the composition with the function $\gamma$.
\end{oss}


Consider a finite set of distinct points $\D=\{D_1,\dots,D_k\}\subset\R^d$ and let $l\ge St(\D)$, where $St(\D)$ is the length of the Steiner set, the minimal among the ones connecting all the points $D_i$ (see \cite[Theorem 4.5.9]{at} for more details on the Steiner problem). Consider the optimization problem:
\be\label{enim}
\min\left\{\E(\Gamma;\V)\ :\ \Gamma\in CMG,\ l(\Gamma)=l,\ \V\subset V(\Gamma),\ \exists\gamma:\Gamma\to\R^d\hbox{ immersion, }\g(\V)=\D\right\},
\ee 
where $CMG$ indicates the class of connected metric graphs. Note that since $l\ge St(\D)$, there is a metric graph and an \emph{embedding} $\gamma:\Gamma\to\R^d$ such that $\D\subset\gamma(V(\Gamma))$ and so the admissible set in the problem \eqref{enim} is non-empty, as well as the admissible set in the problem
\be\label{enem}
\min\left\{\E(\Gamma;\V)\ :\ \Gamma\in CMG,\ l(\Gamma)=l,\ \V\subset V(\Gamma),\ \exists\gamma:\Gamma\to\R^d\hbox{ embedding, }\g(\V)=\D\right\}.
\ee 

\begin{oss}\label{nuovo}
We will see in Theorem \ref{th} that problem \eqref{enim} admits a solution, while Example \ref{example1} shows that in general an optimal embedded graph for problem \eqref{enem} may not exist. In the subsequent Section \ref{s5} we show some explicit examples for which the optimization problem \eqref{enem} admits a solution which is then an embedded graph or equivalently a connected set $C\in\A(\D;l)$. This classical framework is also considered in \cite{fr}, where the author studies the optimization problem \eqref{enem} in the case $\D=\{D_1\}$, corresponding to our Remark \ref{ps}.
\end{oss}

\begin{oss}\label{equivalence}
By Remark \ref{iso} and by the fact that the functionals we consider are invariant with respect to the isometries of the Sobolev space, we have that the problems \eqref{minenergy1} and \eqref{enem} are equivalent, i.e. if $\Gamma\in CMG$ and $\gamma:\Gamma\to\R^d$ is an embedding such that the pair $(\Gamma,\gamma)$ is a solution of \eqref{enem}, then the set $\gamma(\Gamma)$ is a solution of the problem \eqref{minenergy1}. On the other hand, if $C$ is a solution of the problem \eqref{minenergy1}, by Theorem \ref{th1}, we can suppose that $C=\bigcup_{i=1}^N\gamma_i([0,l_i])$, where $\gamma_i$ are injective arc-length parametrized curves, which does not intersect internally. Thus, we can construct a metric graph $\Gamma$ with vertices the set of points $\left\{\g_i(0),\g_i(l_i)\right\}_{i=1}^N\subset\R^d$, and $N$ edges of lengths $l_i$ such that two vertices are connected by an edge, if and only if they are the endpoints of the same curve $\g_i$. The function $\g=(\g_i)_{i=1,\dots,N}:\Gamma\to\R^d$ is an embedding by construction and by Remark \ref{iso}, we have $\E(C;\D)=\E(\Gamma;\D)$.
\end{oss}

\begin{teo}\label{th2}
Let $\D=\{D_1,\dots,D_k\}\subset\R^d$ be a finite set of points and let $l\ge St(\D)$ be a positive real number. Suppose that $\Gamma$ is a connected metric graph of length $l$, $\V\subset V(\Gamma)$ is a set of vertices of $\Gamma$ and $\gamma:\Gamma\to\R^d$ is an immersion (embedding) such that $\D=\gamma(\V)$. Then there exists a connected metric graph $\widetilde\Gamma$ of at most $2k$ vertices and $2k-1$ edges, a set $\widetilde{\V}\subset V(\widetilde\Gamma)$ of vertices of $\widetilde\Gamma$ and an immersion (embedding) $\widetilde\g:\widetilde\Gamma\to\R^d$ such that $\D=\widetilde\gamma(\widetilde{\V})$ and
\be\label{th2e0}
\E(\widetilde\Gamma;\widetilde{\V})\le\E(\Gamma;\V).
\ee
\end{teo}

\begin{proof}
We repeat the argument from Theorem \ref{th1}. We first construct a connected metric graph $\Gamma'$ such that $V(\Gamma')\subset V(\Gamma)$ and the edges of $\Gamma'$ are appropriately chosen paths in $\Gamma$. The edges of $\Gamma$, which are not part of any of these paths, are symmetrized in a single edge, which we attach to $\Gamma'$ in a point, where the restriction of $w$ to $\Gamma'$ achieves its maximum, where $w$ is the energy function for $\Gamma$.

Suppose that $V_1,\dots,V_k\in \V\subset V(\Gamma)$ are such that $\g(V_i)=D_i$, $i=1,\dots,k$. We start constructing $\Gamma'$ by taking $\widetilde{\V}:=\{V_1,\dots,V_k\}\subset V(\Gamma')$. Let $\sigma_1=\{V_{i_0},V_{i_1},\dots, V_{i_s}\}$ be a path of different vertices (i.e. simple path) connecting $V_1=V_{i_s}$ to $V_2=V_{i_0}$ and let $\tilde\sigma_2=\{V_{j_0},V_{j_1},\dots, V_{j_t}\}$ be a simple path connecting $V_1=V_{j_t}$ to $V_3=V_{j_0}$. Let $t'\in\{1,\dots,t\}$ be the smallest integer such that $V_{j_{t'}}\in \sigma_1$. Then we set $V_{j_{t'}}\in V(\Gamma')$ and $\sigma_2=\{V_{j_0},V_{j_1},\dots, V_{j_{t'}}\}$. Consider a simple path $\tilde\sigma_3=\{V_{m_0},V_{m_1},\dots, V_{m_r}\}$ connecting $V_1=V_{m_r}$ to $V_3=V_{m_0}$ and the smallest integer $r'$ such that $V_{m_{r'}}\in \sigma_1\cup\sigma_2$. We set $V_{m_{r'}}\in V(\Gamma')$ and $\sigma_3=\{V_{m_0},V_{m_1},\dots, V_{m_{r'}}\}$. We continue the operation until each of the points $V_1,\dots,V_k$ is in some path $\sigma_j$. Thus we obtain the set of vertices $V(\Gamma')$. We define the edges of $\Gamma'$ by saying that $\{V_i,V_{i'}\}\in E(\Gamma')$ if there is a simple path $\sigma$ connecting $V_i$ to $V_{i'}$ and which is contained in some path $\sigma_j$ from the construction above; the length of the edge $\{V_i,V_{i'}\}$ is the sum of the lengths of the edges of $\Gamma$ which are part of $\sigma$. We notice that $\Gamma'\in CMG$ is a tree with at most $2k-2$ vertices and $2k-2$ edges. Moreover, even if $\Gamma'$ is not a subgraph of $\Gamma$ ($E(\Gamma')$ may not be a subset of $E(\Gamma)$), we have the inclusion $H^1(\Gamma')\subset H^1(\Gamma)$.

Consider the set $E''\subset E(\Gamma)$ composed of the edges of $\Gamma$ which are not part of none of the paths $\sigma_j$ from the construction above. We denote with $l''$ the sum of the lengths of the edges in $E''$. For any $e_{ij}\in E''$ we consider the restriction $w_{ij}:[0,l_{ij}]\to\R$ of the energy function $w$ on $e_{ij}$. Let $v:[0,l'']\to\R$ be the monotone function defined by the equality $|\{v\ge \tau\}|=\sum_{e_{ij}\in E''}|\{w_{ij}\ge\tau\}|$. Using the co-area formula \eqref{coa2} and repeating the argument from Remark \ref{minenergy1p}, we have that 
\be\label{th2e1}
\frac12\int_0^{l''}|v'|^2dx-\int_0^{l''}v(x)\,dx\le\sum_{e_{ij}\in E''}\left(\frac12\int_0^{l_{ij}}|w'_{ij}|^2dx-\int_0^{l_{ij}}w_{ij}\,dx\right).
\ee

Let $\widetilde\Gamma$ be the graph obtained from $\Gamma$ by creating a new vertex $W_1$ in the point, where the restriction $w_{|\Gamma'}$ achieves its maximum, and another vertex $W_2$, connected to $W_1$ by an edge of length $l''$. It is straightforward to check that $\widetilde\Gamma$ is a connected metric tree of length $l$ and that there exists an immersion $\widetilde\gamma:\widetilde\Gamma\to\R^d$ such that $\D=\widetilde\g(\widetilde{\V})$. The inequality \eqref{th2e0} follows since, by \eqref{th2e1}, $J(\widetilde w)\le J(w)$, where $\widetilde w$ is defined as $w$ on the edges $E(\Gamma')\subset E(\widetilde\Gamma)$ and as $v$ on the edge $\{W_1,W_2\}$.
\end{proof}

Before we prove our main existence result, we need a preliminary Lemma.

\begin{lemma}\label{lip}
Let $\Gamma$ be a connected metric tree and let $\V\subset V(\Gamma)$ be a set of Dirichlet vertices. Let $w\in H^1_0(\Gamma;\V)$ be the energy function on $\Gamma$ with Dirichlet conditions in $\V$, i.e. the function that realizes the minimum in the definition of $\E(\Gamma;\V)$. Then, we have the bound $\|w'\|_{L^\infty}\le l(\Gamma)$.  
\end{lemma} 
\begin{proof}
Up to adding vertices in the points where $|w'|=0$, we can suppose that on each edge $e_{ij}:=\{V_i,V_j\}\in E(\Gamma)$ the function $w_{ij}:[0,l_{ij}]\to\R^+$ is monotone. Moreover, up to relabel the vertices of $\Gamma$ we can suppose that if $e_{ij}\in V(\Gamma)$ and $i<j$, then $w(V_i)\le w(V_j)$. Fix $V_i,V_{i'}\in V(\Gamma)$ such that $e_{ii'}\in E(\Gamma)$. Note that, since the derivative is monotone on each edge, it suffices to prove that $|w_{ii'}'(0)|\le l(\Gamma)$. It is enough to consider the case $i<i'$, i.e. $w_{ii'}'(0)>0$. We construct the graph $\widetilde\Gamma$ inductively, as follows (see Figure \ref{f1lip}):
\begin{enumerate}
\item $V_i\in V(\widetilde \Gamma)$;
\item if $V_j\in V(\widetilde\Gamma)$ and $V_k\in V(\Gamma)$ are such that $e_{jk}\in E(\Gamma)$ and $j<k$, then $V_k\in V(\widetilde\Gamma)$ and $e_{jk}\in E(\widetilde\Gamma)$.
\end{enumerate}

\begin{figure}[h]
\centerline{\includegraphics[width=10.0cm]{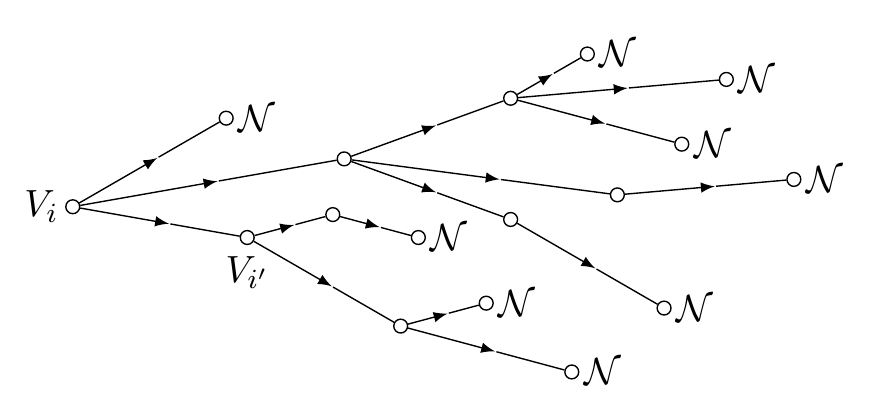}}
\caption{The graph $\widetilde\Gamma$; with the letter $\mathcal{N}$ we indicate the Neumann vertices.}
\label{f1lip}
\end{figure}

The graph $\widetilde\Gamma$ constructed by the above procedure and the restriction $\widetilde w\in H^1(\widetilde\Gamma)$ of $w$ to $\widetilde\Gamma$ have the following properties:
\begin{enumerate}[(a)]
\item On each edge $e_{jk}\in E(\widetilde\Gamma)$, the function $\widetilde w_{jk}$ is non-negative, monotone and $\widetilde w_{jk}''=-1$;
\item $\widetilde w(V_j)>\widetilde w(V_k)$ whenever $e_{jk}\in E(\widetilde\Gamma)$ and $j>k$;
\item if $V_j\in V(\widetilde\Gamma)$ and $j>i$, then there is exactly one $k<j$ such that $e_{kj}\in E(\widetilde\Gamma)$;
\item for $j$ and $k$ as in the previous point, we have that
$$0\le \widetilde w_{kj}'(l_{kj})\le \sum_{s}\widetilde w_{js}'(0),$$
where the sum on the right-hand side is over all $s>j$ such that $e_{sj}\in E(\widetilde\Gamma)$. If there are not such $s$, we have that $\widetilde w_{kj}'(l_{kj})=0$.
\end{enumerate}
The first three conditions follow by the construction of $\widetilde\Gamma$, while condition $(d)$ is a consequence of the Kirchkoff's law for $w$.\\
We prove that for any graph $\widetilde\Gamma$ and any function $\widetilde w\in H^1(\widetilde\Gamma)$, for which the conditions $(a)$, $(b)$, $(c)$ and $(d)$ are satisfied, we have that 
$$\sum_j\widetilde w_{ij}'(0)\le l(\widetilde\Gamma),$$ 
where the sum is over all $j\ge i$ and $e_{ij}\in E(\widetilde\Gamma)$. It is enough to observe that each of the operations $(i)$ and $(ii)$ described below, produces a graph which still satisfies $(a)$, $(b)$, $(c)$ and $(d)$. Let $V_j\in V(\widetilde\Gamma)$ be such that for each $s>j$ for which $e_{js}\in E(\widetilde\Gamma)$, we have that $\widetilde w_{js}'(l_{js})=0$ and let $k<j$ be such that $e_{jk}\in E(\widetilde\Gamma)$.
\begin{enumerate}[(i)]
\item If there is only one $s>j$ with $e_{js}\in E(\widetilde\Gamma)$, then we erase the vertex $V_j$ and the edges $e_{kj}$ and $e_{js}$ and add the edge $e_{ks}$ of length $l_{ks}:=l_{kj}+l_{js}$. On the new edge we define $\widetilde w_{ks}:[0,l_{sk}]\to\R^+$ as
$$\widetilde w_{ks}(x)=-\frac{x^2}{2}+l_{ks}\,x+\widetilde w_{kj}(0),$$
 which still satisfies the conditions above since $\widetilde w_{kj}'-l_{kj}\le l_{js}$, by $(d)$, and $\widetilde w_{ks}'=l_{ks}\ge \widetilde w_{kj}'(0)$.
\item If there are at least two $s>j$ such that $e_{js}\in E(\widetilde\Gamma)$, we erase all the vertices $V_s$ and edges $e_{js}$, substituting them with a vertex $V_S$ connected to $V_j$ by an edge $e_{jS}$ of length 
$$l_{jS}:=\sum_s l_{js},$$
where the sum is over all $s>j$ with $e_{js}\in E(\widetilde\Gamma)$. On the new edge, we consider the function $\widetilde w_{jS}$ defined by 
$$\widetilde w_{jS}(x)=-\frac{x^2}{2}+l_{jS}\,x+\widetilde w(V_j),$$
which still satisfies the conditions above since
$$\sum_{\{s:\,s>j\}} \widetilde w_{js}'(0)=\sum_{\{s:\,s>j\}} l_{js}=l_{jS}=\widetilde w_{jS}'(0).$$ 
\end{enumerate}
We apply $(i)$ and $(ii)$ until we obtain a graph with vertices $V_i,V_j$ and only one edge $e_{ij}$ of length $l(\widetilde\Gamma)$. The function we obtain on this graph is $-\frac{x^2}{2}+l(\widetilde\Gamma)x$ with derivative in $0$ equal to $l(\widetilde\Gamma)$. Since, after applying $(i)$ and $(ii)$, the sum $\sum_{j>i}\widetilde w_{ij}'(0)$ does not decrease, we have the thesis. 
\end{proof}

\begin{teo}\label{th}
Consider a set of distinct points $\D=\{D_1,\dots,D_k\}\subset\R^d$ and a positive real number $l\ge St(\D)$. Then there exists a connected metric graph $\Gamma$, a set of vertices $\V\subset V(\Gamma)$ and an immersion $\gamma:\Gamma\to\R^d$ which are solution of the problem \eqref{enim}. Moreover, $\Gamma$ can be chosen to be a tree of at most $2k$ vertices and $2k-1$ edges.
\end{teo}
 
\begin{proof}
Consider a minimizing sequence $\left(\Gamma_n,\gamma_n\right)$ of connected metric graphs $\Gamma_n$ and immersions $\gamma_n:\Gamma_n\to\R^d$. By Theorem \ref{th2}, we can suppose that each $\Gamma_n$ is a tree with at most $2k$ vertices and $2k-1$ edges. Up to extracting a subsequence, we may assume that the metric graphs $\Gamma_n$ are the same graph $\Gamma$ but with different lengths $l^n_{ij}$ of the edges $e_{ij}$. We can suppose that for each $e_{ij}\in E(\Gamma)$ $l^n_{ij}\to l_{ij}$ for some $l_{ij}\ge0$ as $n\to\infty$. We construct the graph $\widetilde\Gamma$ from $\Gamma$ identifying the vertices $V_i,V_j\in V(\Gamma)$ such that $l_{ij}=0$. The graph $\widetilde\Gamma$ is a connected metric tree of length $l$ and there is an immersion $\widetilde\g:\widetilde\Gamma\to\R^d$ such that $\D\subset\widetilde{\gamma}(\widetilde\Gamma)$. In fact if $\{V_1,\dots V_N\}$ are the vertices of $\Gamma$, up to extracting a subsequence, we can suppose that for each $i=1,\dots,N$ $\gamma_n(V_i)\to X_i\in\R^d$. We define $\widetilde\gamma(V_i):=X_i$ and $\g_{ij}:[0,l_{ij}]\to\R^d$ as any injective arc-length parametrized curve connecting $X_i$ and $X_j$, which exists, since 
$$l_{ij}=\lim l^n_{ij}\ge \lim|\gamma_n(V_i)-\gamma_n(V_j)|=|X_i-X_j|.$$
To prove the theorem, it is enough to check that 
$$\E(\widetilde\Gamma;\V)=\lim_{n\to\infty}\E(\Gamma_n;\V).$$

Let $w^n=(w^n_{ij})_{ij}$ be the energy function on $\Gamma_n$. Up to a subsequence, we may suppose that for each $i=1,\dots, N$, $w^n(V_i)\to a_i\in\R$ as $n\to\infty$. Moreover, by Lemma \ref{lip}, we have that if $l_{ij}=0$, then $a_i=a_j$. On each of the edges $e_{ij}\in E(\widetilde\Gamma)$, where $l_{ij}>0$, we define the function $w_{ij}:[0,l_{ij}]\to\R$ as the parabola such that $w_{ij}(0)=a_i$, $w_{ij}(l_{ij})=a_j$ and $w''_{ij}=-1$ on $(0,l_{ij})$. Then, we have
$$\frac12\int_0^{l^n_{ij}}|(w^n_{ij})'|^2\,dx-\int_0^{l^n_{ij}} w^n_{ij}\,dx\xrightarrow[n\to\infty]{}\frac12\int_0^{l_{ij}}|(w_{ij})'|^2\,dx-\int_0^{l_{ij}} w_{ij}\,dx,$$
and so, it is enough to prove that $\widetilde w=(w_{ij})_{ij}$ is the energy function on $\widetilde\Gamma$, i.e. (by Lemma \ref{enfun}) that the Kirchoff's law holds in each vertex of $\widetilde\Gamma$. This follows since for each $1\le i\ne j\le N$ we have 
\begin{enumerate}
\item $(w^n_{ij})'(0)\to w'_{ij}(0)$, as $n\to\infty$, if $l_{ij}\ne0$;
\item $|(w^n_{ij})'(0)-(w^n_{ij})'(l^n_{ij})|\le l^n_{ij}\to 0$, as $n\to\infty$, if $l_{ij}=0$.
\end{enumerate}
The proof is then concluded.
\end{proof}

The proofs of Theorem \ref{th2} and Theorem \ref{th} suggest that a solution $(\Gamma,\V,\g)$ of the problem \eqref{enim} must satisfy some optimality conditions. We summarize this additional information in the following Proposition. 

\begin{prop}\label{oc}
Consider a connected metric graph $\Gamma$, a set of vertices $\V\subset V(\Gamma)$ and an immersion $\gamma:\Gamma\to\R^d$ such that $(\Gamma,\V,\g)$ is a solution of the problem \eqref{enim}. Moreover, suppose that all the vertices of degree two are in the set $\V$. Then we have that:
\begin{enumerate}[(i)]
\item the graph $\Gamma$ is a tree;
\item the set $\V$ has exactly $k$ elements, where $k$ is the number of Dirichlet points $\{D_1,\dots,D_k\}$;
\item there is at most one vertex $V_j\in V(\Gamma)\setminus\V$ of degree one;
\item if there is no vertex of degree one in $V(\Gamma)\setminus\V$, then the graph $\Gamma$ has at most $2k-2$ vertices and $2k-3$ edges;
\item if there is exactly one vertex of degree one in $V(\Gamma)\setminus\V$, then the graph $\Gamma$ has at most $2k$ vertices and $2k-1$ edges.
\end{enumerate} 
\end{prop}

\begin{proof}
We use the notation $V(\Gamma)=\{V_1,\dots,V_N\}$ for the vertices of $\Gamma$ and $e_{ij}$ for the edges $\{V_i,V_j\}\in E(\Gamma)$, whose lengths are denoted by $l_{ij}$. Moreover, we can suppose that for $j=1,\dots,k$, we have $\g(V_j)=D_j$, where $D_1,\dots,D_k$ are the Dirichlet points from problem \eqref{enim} and so, $\{V_1,\dots,V_k\}\subset\V$. Let $w=(w_{ij})_{ij}$ be the energy function on $\Gamma$ with Dirichlet conditions in the points of $\V$.

\begin{enumerate}
\item Suppose that we can remove an edge $e_{ij}\in E(\Gamma)$, such that the graph $\Gamma'=(V(\Gamma),E(\Gamma)\setminus e_{ij})$ is still connected. Since $w_{ij}''=-1$ on $[0,l_{ij}]$ we have that at least one of the derivatives $w_{ij}'(0)$ and $w_{ij}'(l_{ij})$ is not zero. We can suppose that $w_{ij}'(l_{ij})\neq0$. Consider the new graph $\widetilde\Gamma$ to which we add a new vertex: $V(\widetilde\Gamma)=V(\Gamma)\cup V_0$, then erase the edge $e_{ij}$ and create a new one $e_{i0}=\{V_i,V_0\}$, of the same length, connecting $V_i$ to $V_0$: $E(\widetilde\Gamma)=\left(E(\Gamma)\setminus e_{ij}\right)\cup e_{i0}$. Let $\widetilde w$ be the energy function on $\tilde\Gamma$ with Dirichlet conditions in $\V$. When seen as a subspaces of $\oplus_{ij}H^1([0,l_{ij}])$, we have that $H^1_0(\Gamma;\V)\subset H^1_0(\widetilde\Gamma;\V)$ and so $\E(\widetilde\Gamma;\V)\le\E(\Gamma;\V)$, where the equality occurs, if and only if the energy functions $w$ and $\widetilde w$ have the same components in $\oplus_{ij}H^1([0,l_{ij}])$. In particular, we must have that $w_{ij}=\widetilde{w}_{i0}$ on the interval $[0,l_{ij}]$, which is impossible since $w_{ij}'(l_{ij})\neq0$ and $\widetilde{w}_{i0}'(l_{ij})=0$.

\item Suppose that there is a vertex $V_j\in\V$ with $j>k$ and let $\widetilde w$ be the energy function on $\Gamma$ with Dirichlet conditions in $\{V_1,\dots,V_k\}$. We have the inclusion $H^1_0(\Gamma;\V)\subset H^1_0(\Gamma;\{V_1,\dots,V_k\})$ and so, the inequality $J(\widetilde w)=\E(\Gamma;\{V_1,\dots,V_k\})\le \E(\Gamma;\V)=J(w)$, which becomes an equality if and only if $\widetilde w=w$, which is impossible. Indeed, if the equality holds, then in $V_j$, $w$ satisfies both the Dirichlet condition and the Kirchoff's law. Since $w$ is positive, for any edge $e_{ji}$ we must have $w_{ji}(0)=0$, $w_{ji}'(0)=0$, $w_{ji}''=-1$ ad $w_{ji}\ge0$ on $[0,l_{ji}]$, which is impossible. 

\item Suppose that there are two vertices $V_i$ and $V_j$ of degree one, which are not in $\V$, i.e. $i,j>k$. Since $\Gamma$ is connected, there are two edges, $e_{ii'}$ and $e_{jj'}$ starting from $V_i$ and $V_j$ respectively. Suppose that the energy function $w\in H^1_0(\Gamma;\{V_1,\dots,V_k\})$ is such that $w(V_i)\ge w(V_j)$. We define a new graph $\tilde\Gamma$ by erasing the edge $e_{jj'}$ and creating the edge $e_{ij}$ of length $l_{jj'}$. On the new edge $e_{ij}$ we consider the function $w_{ij}(x)=w_{jj'}(x)+w(V_i)-w(V_j)$. The function $\widetilde w$ on $\widetilde\Gamma$ obtained by this construction is such that $J(\widetilde w)\le J(w)$, which proves the conclusion.
\end{enumerate}
The points $(iv)$ and $(v)$ follow by the construction in Theorem \ref{th2} and the previous claims $(i)$, $(ii)$ and $(iii)$.
\end{proof}

\begin{oss}
Suppose that $V_j\in V(\Gamma)\setminus\V$ is a vertex of degree one and let $V_i$ be the vertex such that $e_{ij}\in E(\Gamma)$. Then the energy function $w$ with Dirichlet conditions in $\V$ satisfies $w_{ji}'(0)=0$. In this case, we call $V_j$ a Neumann vertex. By Proposition \ref{oc}, an optimal graph has at most one Neumann vertex. 
\end{oss}

\section{Some examples of optimal metric graphs}\label{s5}

In this section we show three examples. In the first one we deal with two Dirichlet points, the second concerns three aligned Dirichlet points and the third one deals with the case in which the Dirichlet points are vertices of an equilateral triangle. In the first and the third one we find the minimizer explicitly as an embedded graph, while in the second one we limit ourselves to prove that there is no embedded minimizer of the energy, i.e. the problem \eqref{enem} does not admit a solution.

In the following example we use a symmetrization technique similar to the one from Remark \ref{ps}.

\begin{exam}\label{example2}
Let $D_1$ and $D_2$ be two distinct points in $\R^d$ and let $l\ge|D_1-D_2|$ be a real number. Then the problem
\be\label{enim2}
\begin{array}{ll}
\min\lbrace\E(\Gamma;\{V_1,V_2\}): \Gamma\in CMG,\ l(\Gamma)=l,\ V_1,V_2\in V(\Gamma),\\
\qquad\qquad\qquad \hbox{exists}\ \gamma:\Gamma\to\R\ \hbox{immersion},\ \g(V_1)=D_1, \g(V_2)=D_2\rbrace.
\end{array}
\ee
has a solution $(\Gamma,\gamma)$, where $\Gamma$ is a metric graph with vertices $V(\Gamma)=\{V_1,V_2,V_3,V_4\}$ and edges $E(\Gamma)=\{e_{13}=\{V_1,V_3\},e_{23}=\{V_2,V_3\},e_{43}=\{V_4,V_3\}\}$ of lengths $l_{13}=l_{23}=\frac12|D_1-D_2|$ and $l_{34}=l-|D_1-D_2|$, respectively. The map $\gamma:\Gamma\to\R^d$ is an embedding such that $\gamma(V_1)=D_1$, $\gamma(V_2)=D_2$ and $\gamma(V_3)=\frac{D_1+D_2}{2}$ (see Figure \ref{f1ex2}).

\begin{figure}[h]
\centerline{\includegraphics[width=10.0cm]{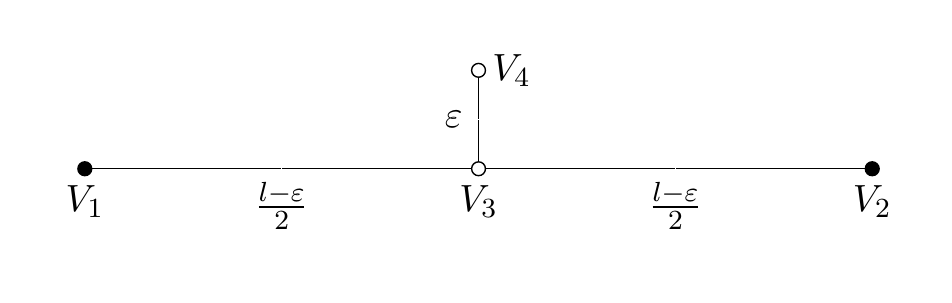}}
\caption{The optimal graph with two Dirichlet points.}
\label{f1ex2}
\end{figure}

To fix the notations, we suppose that $|D_1-D_2|=l-\varepsilon$. Let $u=(u_{ij})_{ij}$ be the energy function of a generic metric graph $\Sigma$ and immersion $\sigma:\Sigma\to\R^d$ with $D_1,D_2\in\sigma(V(\Sigma))$. Let $M=\max\{u(x):\,x\in\Sigma\}>0$. We construct a candidate $v\in H^1_0(\Gamma;\{V_1,V_2\})$ such that $J(v)\le J(u)$, which immediately gives the conclusion.

We define $v$ by the following three \emph{increasing} functions
$$v_{13}=v_{23}\in H^1([0,(l-\varepsilon)/2]),\qquad v_{34}\in H^1([0,\varepsilon]),$$
with boundary values
$$v_{13}(0)=v_{23}(0)=0,\qquad v_{13}((l-\varepsilon)/2)=v_{23}((l-\varepsilon)/2)=v_{34}(0)=m<M,$$
and level sets uniquely determined by the equality $\ \mu_u=\mu_v$, where $\mu_u$ and $\mu_v$ are the distribution functions of $u$ and $v$ respectively, defined by
\begin{align*}
&\mu_u(t)=\H^1(\{u\le t\})=\sum_{e_{ij}\in E(\Sigma)}\H^1(\{u_{ij}\le t\}),\\
&\mu_v(t)=\H^1(\{v\le t\})=\sum_{j=1,2,4}\H^1(\{v_{j3}\le t\}).
\end{align*}
As in Remark \ref{ps} we have $\|v\|_{L^1(\Gamma)}=\|u\|_{L^1(C)}$ and
\be\label{e1l3}
\int_\Sigma |u'|^2\,dx
=\int_0^M\!\!\!\Big(\sum_{u=\tau}|u'|\Big)\,d\tau
\ge\int_0^M\!\!n^2_u(\tau)\left(\sum_{u=\tau}\frac{1}{|u'|(\tau)}\right)^{-1}\!\!\!d\tau
=\int_0^M\frac{n_u^2(\tau)}{\mu_u'(\tau)}\,d\tau
\ee
where $n_u(\tau)=\H^0(\{u=\tau\})$. The same argument holds for $v$ on the graph $\Gamma$ but, this time, with the equality sign:
\be\label{e3l3}
\int_\Gamma|v'|^2dx=\int_0^M\!\!\Big(\sum_{v=\tau}|v'|\Big)\,d\tau=\int_0^M\frac{n_v^2(\tau)}{\mu_v'(\tau)}\,d\tau,
\ee
since $|v'|$ is constant on $\{v=\tau\}$, for every $\tau$.
Then, in view of \eqref{e1l3} and \eqref{e3l3}, to conclude it is enough to prove that $n_u(\tau)\ge n_v(\tau)$ for almost every $\tau$. To this aim we first notice that, by construction $n_v(\tau)=1$ if $\tau\in [m,M]$ and $n_v(\tau)=2$ if $\tau\in [0,m)$. Since $n_u$ is decreasing and greater than $1$ on $[0,M]$, we only need to prove that $n_u\ge2$ on $[0,m]$. To see this, consider two vertices $W_1,W_2\in V(\Sigma)$ such that $\sigma(W_1)=D_1$ and $\sigma(W_2)=D_2$. Let $\eta$ be a simple path connecting $W_1$ to $W_2$ in $\Sigma$. Since $\sigma$ is an immersion we know that the length $l(\eta)$ of $\eta$ is at least $l-\varepsilon$. By the continuity of $u$, we know that $n_u\ge 2$ on the interval $[0,\max_\eta u)$. Since $n_v=1$ on $[m,M]$, we need to show that $\max_\eta u\ge m$. Otherwise, we would have 
$$l(\eta)\le |\{u\le \max_\eta u\}|<|\{u\le m\}|=|\{v\le m\}|=|D_1-D_2|\le l(\eta),$$
which is impossible.
\end{exam}

\begin{oss}\label{max}
In the previous example the optimal metric graph $\Gamma$ is such that for any (admissible) immersion $\gamma:\Gamma\to\R^d$, we have $|\g(V_1)-\g(V_3)|=l_{13}$ and $|\g(V_2)-\g(V_3)|=l_{23}$, i.e. the point $\g(V_3)$ is necessary the midpoint $\frac{D_1+D_2}{2}$, so we have a sort of \emph{rigidity} of the graph $\Gamma$. More generally, we say that an edge $e_{ij}$ is \emph{rigid}, if for any admissible immersion $\g:\Gamma\to\R^d$, i.e. an immersion such that $\D=\gamma(\V)$, we have $|\g(V_i)-\g(V_j)|=l_{ij}$, in other words the realization of the edge $e_{ij}$ in $\R^d$ via any immersion $\g$ is a segment. One may expect that in the optimal graph all the edges, except the one containing the Neumann vertex, are rigid. Unfortunately, we are able to prove only the weaker result that:
\begin{enumerate}
\item if the energy function $w$, of an optimal metric graph $\Gamma$, has a local maximum in the interior of an edge $e_{ij}$, then the edge is rigid; if the maximum is global, then $\Gamma$ has no Neumann vertices;
\item if $\Gamma$ contains a Neumann vertex $V_j$, then $w$ achieves its maximum at it.
\end{enumerate} 
To prove the second claim, we just observe that if it is not the case, then we can use an argument similar to the one from point $(iii)$ of Proposition \ref{oc}, erasing the edge $e_{ij}$ containing the Neumann vertex $V_j$ and creating an edge of the same length that connects $V_j$ to the point, where $w$ achieves its maximum, which we may assume a vertex of $\Gamma$ (possibly of degree two).

For the first claim, we apply a different construction which involves a symmetrization technique. In fact, if the edge $e_{ij}$ is not rigid, then we can create a new metric graph of smaller energy, for which there is still an immersion which satisfies the conditions in problem \eqref{enim}. In this there are points $0<a<b<l_{ij}$ such that $l_{ij}-(b-a)\ge |\gamma(V_i)-\gamma(V_j)|$ and $\min_{[a,b]}w_{ij}=w_{ij}(a)=w_{ij}(b)<\max_{[a,b]}w_{ij}$. Since the edge is not rigid, there is an immersion $\gamma$ such that $|\g_{ij}(a)-\g_{ij}(b)|>|b-a|$. The problem \eqref{enim2} with $D_1=\gamma_{ij}(a)$ and $D_2=\gamma_{ij}(b)$ has as a solution the $T$-like graph described in Example \ref{example2}. This shows, that the original graph could not be optimal, which is a contradiction.
\end{oss}

\begin{exam}\label{example1}
Consider the set of points $\D=\{D_1, D_2, D_3\}\subset\R^2$ with coordinates respectively $(-1,0)$, $(1,0)$ and $(n,0)$, where $n$ is a positive integer. Given $l=(n+2)$, we aim to show that for $n$ large enough there is no solution of the optimization problem 
\be
\min\left\{\E(\Gamma;\V): \Gamma\in CMG,\ l(\Gamma)=l,\ \V\subset V(\Gamma),\ \exists\gamma:\Gamma\to\R\ \hbox{embedding},\ \D=\g(\V)\right\}.
\ee
In fact, we show that all the possible solutions of the problem 
\be\label{enim3}
\min\left\{\E(\Gamma;\V): \Gamma\in CMG,\ l(\Gamma)=l,\ \V\subset V(\Gamma),\ \exists\gamma:\Gamma\to\R\ \hbox{immersion},\ \D=\g(\V)\right\}
\ee
are metric graphs $\Gamma$ for which there is no embedding $\gamma:\Gamma\to\R^2$ such that $\D\subset\gamma(V(\Gamma))$. Moreover, there is a sequence of embedded metric graphs which is a minimizing sequence for the problem \eqref{enim3}.

More precisely, we show that the only possible solution of \eqref{enim3} is one of the following metric trees:
\begin{enumerate}[(i)]
\item $\Gamma_1$ with vertices $V(\Gamma_1)=\{V_1,V_2,V_3,V_4\}$ and edges $E(\Gamma_1)=(e_{14}=\{V_1,V_4\},e_{24}=\{V_2,V_4\},e_{34}=\{V_3,V_4\}$ of lengths $l_{14}=l_{24}=1$ and $l_{34}=n$, respectively. The set of vertices in which the Dirichlet condition holds is $\V_1=\{V_1,V_2,V_3\}$. 
\item $\Gamma_2$ with vertices $V(\Gamma_2)=\{W_i\}_{i=1}^6$, and edges $E(\Gamma_2)=\{e_{14},e_{24},e_{35},e_{45},e_{56}\}$ ,where $e_{ij}=\{W_i,W_j\}$ for $1\le i\ne j\le6$ of lengths $l_{14}=1+\alpha,\,l_{24}=1-\alpha,\,l_{35}=n-\beta,\,l_{45}=\beta-\alpha$, $l_{56}=\alpha$, where $0<\alpha<1$ and $\alpha<\beta<n$. The set of vertices in which the Dirichlet condition holds is $\V_1=\{V_1,V_2,V_3\}$. A possible immersion $\gamma$ is described in Figure \ref{f1ex1}.
\end{enumerate}

\begin{figure}[h]
\centerline{\includegraphics[width=14.0cm]{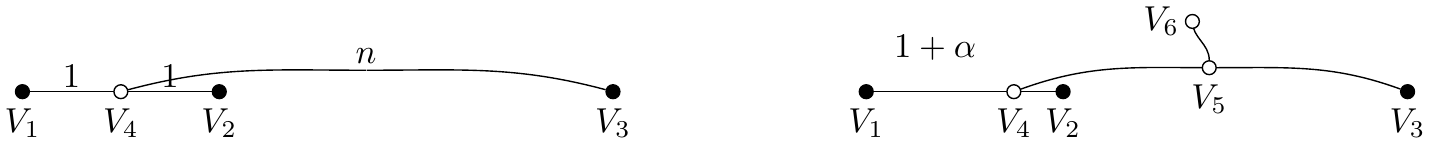}}
\caption{The two candidates for a solution of \eqref{enim3}.}
\label{f1ex1}
\end{figure}

We start showing that if there is an optimal metric graph with no Neumann vertex, then it must be $\Gamma_1$. In fact, by Proposition \ref{oc}, we know that the optimal metric graph is of the form $\Gamma_1$, but we have no information on the lengths of the edges, which we set as $l_i=l(e_{i4})$, for $i=1,2,3$ (see Figure \ref{f2ex1}). We can calculate explicitly the minimizer of the energy functional and the energy itself in function of $l_1$, $l_2$ and $l_3$. 

\begin{figure}[h]
\centerline{\includegraphics[width=7.0cm]{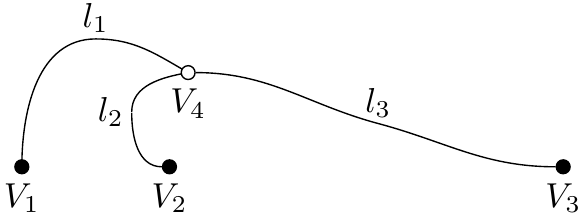}}
\caption{A metric tree with the same topology as $\Gamma_1$.}
\label{f2ex1}
\end{figure}

The minimizer of the energy $w:\Gamma\to\R$ is given by the functions $w_i:[0,l_i]\to\R$, where $i=1,2,3$ and
\be
w_i(x)=-\frac{x^2}{2}+a_ix.
\ee
where 
\be
a_1=\frac{l_1}{2}+\frac{l_2l_3(l_1+l_2+l_3)}{2(l_1l_2+l_2l_3+l_3l_1)},
\ee
and $a_2$ and $a_3$ are defined by a cyclic permutation of the indices. As a consequence, we obtain that the derivative along the edge $e_{14}$ in the vertex $V_4$ is given by
\be\label{der}
w_1'(l_1)=-l_1+a_1=-\frac{l_1}{2}+\frac{l_2l_3(l_1+l_2+l_3)}{2(l_1l_2+l_2l_3+l_3l_1)},
\ee
and integrating the energy function $w$ on $\Gamma$, we obtain
\be
\E(\Gamma;\{V_1,V_2,V_3\})=-\frac{1}{12}(l_1^3+l_2^3+l_3^3)-\frac{(l_1+l_2+l_3)^2 l_1l_2l_3}{4(l_1l_2+l_2l_3+l_3l_1)}.
\ee
Studying this function using Lagrange multipliers is somehow complicated due to the complexity of its domain. Thus we use a more geometric approach applying the symmetrization technique described in Remark \ref{ps} in order to select the possible candidates. We prove that if the graph is optimal, then all the edges must be rigid (this would force the graph to coincide with $\Gamma_1$). Suppose that the optimal graph $\Gamma$ is not rigid, i.e. there is a non-rigid edge. Then, for $n>4$, we have that $l_2<l_1<l_3$ and so, by \eqref{der}, we obtain $w_3'(l_3)< w_1'(l_1)< w_2'(l_2)$. As a consequence of the Kirchoff's law we have $w_3'(l_3)<0$ and $w_2'(l_2)>0$ and so, $w$ has a local maximum on the edge $e_{34}$ and is increasing on $e_{14}$. By Remark \ref{max}, we obtain that the edge $e_{34}$ is rigid.

We first prove that $w_1'(l_1)>0$. In fact, if this is not the case, i.e. $w_1'(l_1)<0$, by Remark \ref{max}, we have that the edges $e_{14}$ is also rigid and so, $l_1+l_3=|D_1-D_3|=n+1$, i.e. $l_2=1$. Moreover, by \eqref{der}, we have that $w_1'(l_1)<0$, if and only if $l_1^2>l_2l_3=l_3$. The last inequality does not hold for $n>11$, since, by the triangle inequality, $l_2+l_3\ge|D_2-D_3|=n-1$, we have $l_1\le 3$. Thus, for $n$ large enough, we have that $w$ is increasing on the edge $e_{14}$.

We now prove that the edges $e_{14}$ and $e_{24}$ are rigid. In fact, suppose that $e_{24}$ is not rigid. Let $a\in (0,l_1)$ and $b\in(0,l_2)$ be two points close to $l_1$ and $l_2$ respectively and such that $w_{14}(a)=w_{24}(b)<w(V_4)$ since $w_{14}$ and $w_{24}$ are strictly increasing. Consider the metric graph $\widetilde\Gamma$ whose vertices and edges are $$V(\widetilde\Gamma)=\{V_1=\widetilde{V}_1,\,V_2=\widetilde{V}_2,\,V_3=\widetilde{V}_3,\,V_4=\widetilde{V}_4,\,\widetilde{V}_5,\,\widetilde{V}_6\},$$ 
$$E(\widetilde\Gamma)=\{e_{15},\,e_{25},\,e_{45},\,e_{34},\,e_{46}\},$$
where $e_{ij}=\{\widetilde{V}_i,\widetilde{V}_j\}$ and the lengths of the edges are respectively (see Figure \ref{f3ex1})
$$\widetilde{l}_{15}=a,\ \widetilde{l}_{25}=b,\ \widetilde{l}_{45}=l_2-b,\ \widetilde{l}_{34}=l_{3},\ \widetilde{l}_{46}=l_1-a.$$

\begin{figure}[h]
\centerline{\includegraphics[width=14.0cm]{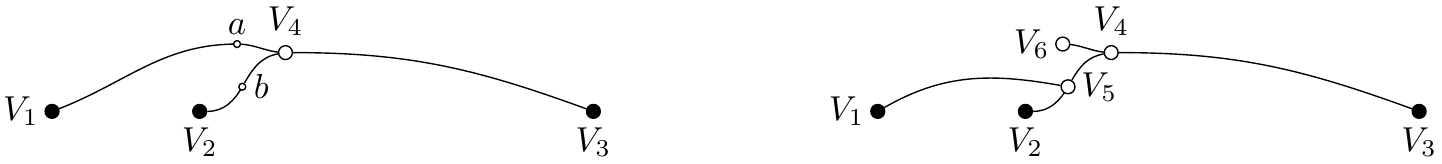}}
\caption{The graph $\Gamma$ (on the left) and the modified one $\widetilde\Gamma$ (on the right).}
\label{f3ex1}
\end{figure}

The new metric graph is still a competitor in the problem \eqref{enim3} and there is a function $w\in H^1_0(\widetilde\Gamma;\{V_1,V_2,V_3\})$ such that $\E(\widetilde\Gamma;\{V_1,V_2,V_3\})< J(\widetilde w)=J(w)$, which is a contradiction with the optimality of $\Gamma$. In fact, it is enough to define $\widetilde w$ as 
$$\widetilde{w}_{15}=w_{14}\vert_{[0,a]},\ \widetilde{w}_{25}=w_{24}\vert_{[0,b]},\ \widetilde{w}_{54}=w_{24}\vert_{[b,l_2]},\ \widetilde{w}_{34}=w_{34},\ \widetilde{w}_{64}=w_{14}\vert_{[a,l_1]},$$
and observe that $\widetilde w$ is not the energy function on the graph $\widetilde\Gamma$ since it does not satisfy the Neumann condition in $\widetilde{V}_6$. In the same way, if we suppose that $w_{14}$ is not rigid, we obtain a contradiction, and so all the three edges must be rigid, i.e. $\Gamma=\Gamma_1$.

In a similar way we prove that a metric graph $\Gamma$ with a Neumann vertex can be a solution of \eqref{enim3} only if it is of the same form as $\Gamma_2$. We proceed in two steps: first, we show that, for $n$ large enough, the edge containing the Neumann vertex has a common vertex with the longest edge of the graph; then we can conclude reasoning analogously to the previous case. Let $\Gamma$ be a metric graph with vertices $V(\Gamma)=\{V_i\}_{i=1}^6$, and edges $E(\Gamma)=\{e_{15},e_{24},e_{34},e_{45},e_{56}\}$, where $e_{ij}=\{V_i,V_j\}$ for $1\le i\neq j\le6$.

We prove that $w(V_6)\le \max_{e_{34}}w$, i.e. the graph $\Gamma$ is not optimal, since, by Remark \ref{max}, the maximum of $w$ must be achieved in the Neumann vertex $V_6$ (the case $E(\Gamma)=\{e_{14},e_{25},e_{34},e_{45},e_{56}\}$ is analogous). Let $w_{15}:[0,l_{15}]\to\R$, $w_{65}:[0,l_{65}]\to\R$ and $w_{34}:[0,l_{34}]\to\R$ be the restrictions of the energy function $w$ of $\Gamma$ to the edges $e_{15}$, $e_{65}$ and $e_{34}$ of lengths $l_{15}$, $l_{65}$ and $l_{34}$, respectively. Let $u:[0,l_{15}+l_{56}]\to\R$ be defined as
\be
u(x)=\begin{cases}
w_{15}(x),\ x\in[0,l_{15}],\\
w_{56}(x-l_{15}),\ x\in [l_{15}.l_{15}+l_{56}].
\end{cases}
\ee
If the metric graph $\Gamma$ is optimal, then the energy function on $w_{54}$ on the edge $e_{45}$ must be decreasing and so, by the Kirchhoff's law in the vertex $V_5$, we have that $w_{15}'(l_{15})+w_{65}'(l_{65})\le 0$, i.e. the left derivative of $u$ at $l_{15}$ is less than the right one:
$$\partial_- u(l_{15})=w_{15}'(l_{15})\le w_{56}'(0)=\partial_+u(l_{15}).$$
By the maximum principle, we have that 
$$u(x)\le\widetilde{u}(x)=-\frac{x^2}{2}+(l_{15}+l_{56})x\le \frac12(l_{15}+l_{56})^2.$$
On the other hand, $w_{34}(x)\ge v(x)=-\frac{x^2}{2}+\frac{l_{34}}{2}x$, again by the maximum principle on the interval $[0,l_{34}]$. Thus we have that 
$$\max_{x\in[0,l_{34}]}\,w_{34}(x)\ge \max_{x\in[0,l_{34}]}\,v(x)=\frac18l_{34}^2>\frac12(l_{15}+l_{56})^2\ge w(V_6),$$
for $n$ large enough.

Repeating the same argument, one can show that the optimal metric graph $\Gamma$ is not of the form $V(\Gamma)=(V_1,V_2,V_3,V_4,V_5)$, $E(\Gamma)=\{V_1,V_4\}, \{V_2,V_4\}, \{V_3,V_4\}, \{V_4,V_5\}$.

Thus, we obtained that the if the optimal graph has a Neumann vertex, then the corresponding edge must be attached to the longest edge. To prove that it is of the same form as $\Gamma_2$, there is one more case to exclude, namely: $\Gamma$ with vertices, $V(\Gamma)=(V_1,V_2,V_3,V_4,V_5)$, $E(\Gamma)=\left\{\{V_1,V_2\}, \{V_2,V_4\}, \{V_3,V_4\}, \{V_4,V_5\}\right\}$ (see Figure \ref{f4ex1}). By Example \ref{example2}, the only possible candidate of this form is the graph with lengths $l(\{V_1,V_2\})=|D_1-D_2|=2$, $l(\{V_2,V_4\})=\frac{n-1}{2}$, $l(\{V_3,V_4\})=\frac{n-1}{2}$, $l(\{V_4,V_5\})=2$. In this case, we compare the energy of $\Gamma$ and $\Gamma_1$, by an explicit calculation:
\be
\E(\Gamma;\{V_1,V_2,V_3\})=-\frac{n^3-3n^2+6n}{24}>-\frac{n^2(n+1)^2}{12(2n+1)}=\E(\Gamma_1;\{V_1,V_2,V_3\}),
\ee
for $n$ large enough. 
\end{exam}

\begin{figure}[h]
\centerline{\includegraphics[width=15.0cm]{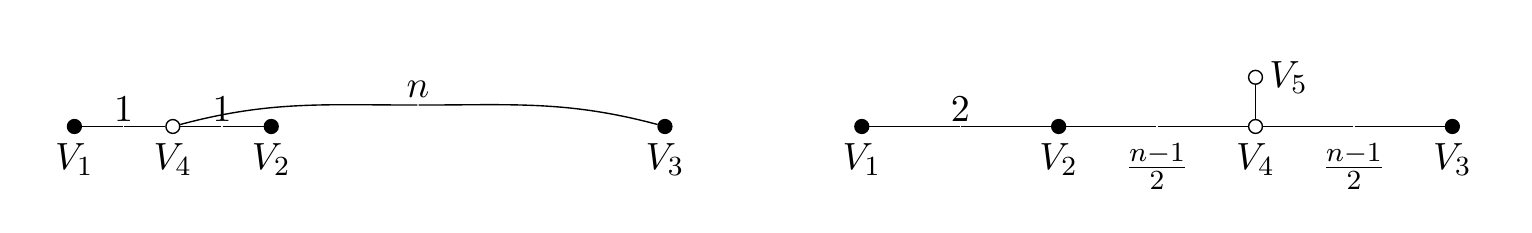}}
\caption{The graph $\Gamma_1$ (on the left) has lower energy than the graph $\Gamma$ (on the right).}
\label{f4ex1}
\end{figure}

Before we pass to our last example, we need the following Lemma.

\begin{lemma}\label{att}
Let $w_a:[0,1]\to\R$ be given by $w_a(x)=-\frac{x^2}{2}+ax$, for some positive real number $a$. If $w_a(1)\le w_A(1)\le \max_{x\in[0,1]}\,w_a(x)$, then $J(w_A)\le J(w_a)$, where $J(w)=\frac{1}{2}\int_0^1|w'|^2\,dx-\int_0^1w\,dx.$ 
\end{lemma}

\begin{proof}
It follows by performing the explicit calculations.
\end{proof}

\begin{exam}\label{example3}
Let $D_1$, $D_2$ and $D_3$ be the vertices of an equilateral triangle of side $1$ in $\R^2$, i.e.
$$D_1=(-\frac{\sqrt3}{3},0),\ D_2=(\frac{\sqrt3}{6},-\frac12),\ D_3=(\frac{\sqrt3}{6},\frac12).$$
We study the problem \eqref{enim} with $\D=\{D_1,D_2,D_3\}$ and $l>\sqrt 3$. We show that the solutions may have different qualitative properties for different $l$ and that there is always a symmetry breaking phenomena, i.e. the solutions does not have the same symmetries as the initial configuration $\D$. We first reduce our study to the following three candidates (see Figure \ref{f1ex3}):

\begin{enumerate}
\item The metric tree $\Gamma_1$, defined by with vertices $V(\Gamma)=\{V_1,V_2,V_3,V_4\}$ and edges $E(\Gamma)=\{e_{14}, e_{24}, e_{34}\}$, where $e_{ij}=\{V_i,V_j\}$ and the lengths of the edges are respectively $l_{24}=l_{34}=x$, $l_{14}=\frac{\sqrt3}{2}-\sqrt{x^2-\frac14}$, for some $x\in [1/2,1/\sqrt3]$. Note that the length of $\Gamma_1$ is less than $1+\sqrt3/2$, i.e. it is a possible solution only for $l\le 1+\sqrt3/2$. The new vertex $V_4$ is of Kirchhoff type and there are no Neumann vertices.

\item The metric tree $\Gamma_2$ with vertices $V=(V_1,V_2,V_3,V_4,V_5)$ and $E(\Gamma)=\{e_{14}, e_{24}, e_{34},e_{45}\}$, where $e_{ij}=\{V_i,V_j\}$ and the lengths of the edges $l_{14}=l_{24}=l_{34}=1/\sqrt3$, $l_{45}=l-\sqrt3$, respectively. The new vertex $V_4$ is of Kirchhoff type and $V_5$ is a Neumann vertex.

\item The metric tree $\Gamma_3$ with vertices $V(\Gamma)=\{V_1,V_2,V_3,V_4,V_5,V_6\}$ and edges $E(\Gamma)=\{e_{15},e_{24},e_{34},e_{45},e_{56}\}$, where $e_{ij}=\{V_i,V_j\}$ and the lengths of the edges are $l_{24}=l_{34}=x$, $l_{15}=\frac{lx}{2(2l-3x)}+\frac{\sqrt3}{4}-\frac14\sqrt{4x^2-1}$, $l_{45}=\frac{\sqrt3}{4}-\frac{lx}{2(2l-3x)}-\frac14\sqrt{4x^2-1}$ and $l_{56}=l-2x-\sqrt3/2+\frac12\sqrt{4x^2-1}$. The new vertices $V_4$ and $V_5$ are of Kirchhoff type and $V_6$ is a Neumann vertex.
\end{enumerate}

\begin{figure}[h]
\centerline{\includegraphics[width=14.0cm]{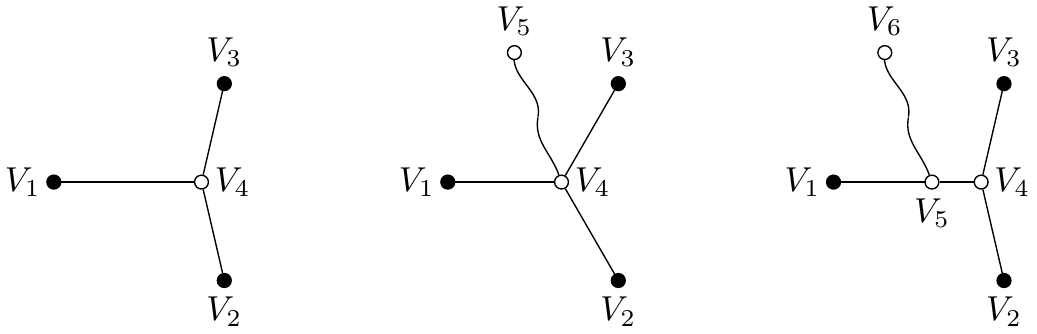}}
\caption{The three competing graphs.}
\label{f1ex3}
\end{figure}

Suppose that the metric graph $\Gamma$ is optimal and has the same vertices and edges as $\Gamma_1$. Without loss of generality, we can suppose that the maximum of the energy function $w$ on $\Gamma$ is achieved on the edge $e_{14}$. If $l_{24}\neq l_{34}$, we consider the metric graph $\widetilde\Gamma$ with the same vertices and edges as $\Gamma$ and lengths $\widetilde l_{14}=l_{14}$, $\widetilde l_{24}=\widetilde l_{34}=(l_{24}+l_{34})/2$. An immersion $\widetilde\g:\widetilde\Gamma\to\R^2$, such that $\widetilde\g(V_j)=D_j$, for $j=1,2,3$ still exists and the energy decreases, i.e. $\E(\widetilde\Gamma;\{V_1,V_2,V_3\})<\E(\Gamma;\{V_1,V_2,V_3\})$. In fact, let $v=\widetilde w_{24}=\widetilde w_{34}:[0,\frac{l_{24}+l_{34}}{2}]\to\R$ be an increasing function such that $2|\{v\ge \tau\}|=|\{w_{24}\ge\tau\}|+|\{w_{34}\ge\tau\}|$. By the classical Polya-Szego inequality and by the fact that $w_{24}$ and $w_{34}$ have no constancy regions, we obtain that
$$J(\widetilde w_{24})+J(\widetilde w_{34})<J(w_{24})+J(w_{34}),$$
and so it is enough to construct a function $\widetilde w_{14}:[0,l_{14}]\to\R$ such that $\widetilde w_{14}(l_{14})=\widetilde w_{24}=\widetilde w_{34}$ and $J(\widetilde w_{14})\le J(w_{14})$. Consider a function such that $\widetilde w_{14}''=-1$, $\widetilde w_{14}(0)=0$ and $\widetilde w_{14}(l_{14})=\widetilde w_{24}(l_24)=\widetilde w_{34}(l_{34})$. Since we have the inequality $w_{14}(l_{14})\le\widetilde w_{14}(l_{14})\le\max_{[0,l_{14}]}\,w_{14}=\max_{\Gamma}\,w$, we can apply Lemma \ref{att} and so, $J(\widetilde w_{14})\le J(w_{14})$. Thus, we obtain that $l_{24}=l_{34}$ and that both the functions $w_{24}$ and $w_{34}$ are increasing (in particular, $l_{14}\ge l_{24}=l_{34}$). If the maximum of $w$ is achieved in the interior of the edge $e_{14}$ then, by Remark \ref{max}, the edge $e_{14}$ must be rigid and so, all the edges must be rigid. Thus, $\Gamma$ coincides with $\Gamma_1$ for some $x\in (\frac12,\frac{1}{\sqrt3}]$. If the maximum of $w$ is achieved in the vertex $V_4$, then applying one more time the above argument, we obtain $l_{14}=l_{24}=l_{34}=\frac{1}{\sqrt3}$, i.e. $\Gamma$ is $\Gamma_1$ corresponding to $x=\frac{1}{\sqrt3}$.

Suppose that the metric graph $\Gamma$ is optimal and that has the same vertices as $\Gamma_2$. If $w=(w_{ij})_{ij}$ is the energy function on $\Gamma$ with Dirichlet conditions in $\{V_1,V_2,V_3\}$, we have that $w_{14},w_{24}$ an $w_{34}$ are increasing on the edges $e_{14}, e_{24}$ and $e_{34}$. As in the previous situation $\Gamma=\Gamma_1$, by a symmetrization argument, we have that $l_{14}=l_{24}=l_{34}$. Since any level set $\{w=\tau\}$ contains exactly $3$ points, if $\tau<w(V_4)$, and $1$ point, if $\tau\ge w(V_4)$, we can apply the same technique as in Example \ref{example2} to obtain that $l_{14}=l_{24}=l_{34}=\frac{1}{\sqrt3}$.

Suppose that the metric graph $\Gamma$ is optimal and that has the same vertices and edges as $\Gamma_3$. Let $w$ be the energy function on $\Gamma$ with Dirichlet conditions in $\{V_1,V_2,V_3\}$. Since we assume $\Gamma$ optimal, we have that $w_{45}$ is increasing on the edge $e_{45}$ and $w(V_5)\ge w_{ij}$, for any $\{i,j\}\neq\{5,6\}$. Applying the symmetrization argument from the case $\Gamma=\Gamma_1$ and Lemma \ref{att}, we obtain that $l_{24}=l_{34}=x$ and that the functions $w_{24}=w_{34}$ are increasing on $[0,l_{24}]$. Let $a\in[0,l_{15}]$ be such that $w_{15}(a)=w(V_4)$. By a symmetrization argument, we have that necessarily $l_{15}-a=l_{45}$ an that $w_{45}(x)=w_{15}(x-a)$. Moreover, the edges $e_{15}$ and $e_{45}$ are rigid. Indeed, for any admissible immersion $\gamma=(\g_{ij})_{ij}:\Gamma\to\R^2$, we have that the graph $\widetilde\Gamma$ with vertices $V(\widetilde\Gamma)=\{\widetilde V_1,V_4,V_5,V_6\}$ and edges $E(\widetilde\Gamma)=\left\{\{\widetilde V_1,V_5\},\{V_4,V_5\},\{V_5,V_6\}\right\}$, is a solution for the problem \eqref{enim2} with $D_1:=\gamma_{15}(a)$ and $D_2:=\g(V_4)$. By Example \ref{example2} and Remark \ref{max}, we have $|\gamma_{15}(a)-\g(V_4)|=2l_{45}$ and, since this holds for every admissible $\g$, we deduce the rigidity of $e_{15}$ and $e_{45}$. Using this information one can calculate explicitly all the lengths of the edges of $\Gamma$ using only the parameter $x$, obtaining the third class of possible minimizers.

\begin{figure}[h]
\centerline{\includegraphics[width=14.0cm]{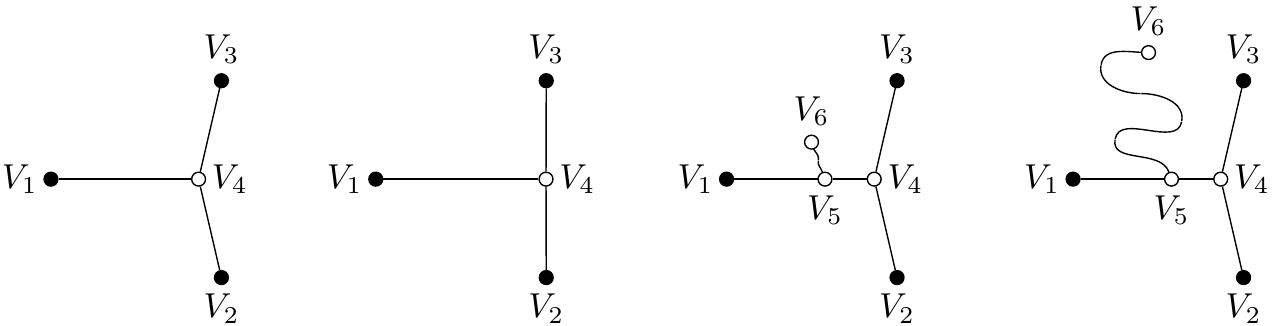}}
\caption{The optimal graphs for $l<1+\sqrt{3}/2$, $l=1+\sqrt3/2$, $l>1+\sqrt3/2$ and $l>>1+\sqrt3/2$.}
\label{f2ex3}
\end{figure}

An explicit estimate of the energy shows that:
\begin{enumerate}
\item If $\sqrt3\le l\le 1+\sqrt3/2$, we have that the solution of the problem \eqref{enim} with $\D=\{D_1,D_2,D_3\}$ is of the form $\Gamma_1$ (see Figure \ref{f2ex3}).
\item If $l>1+\sqrt3/2$, then the solution of the problem \eqref{enim} with $\D=\{D_1,D_2,D_3\}$ is of the form $\Gamma_3$.
\end{enumerate}

In both cases,the parameter $x$ is uniquely determined by the total length $l$ and so, we have uniqueness up to rotation on $\frac{2\pi}{3}$. Moreover, in both cases the solutions are metric graphs, for which there is an embedding $\gamma$ with $\gamma(V_i)=D_i$, i.e. they are also solutions of the problem \eqref{enem} with $\D=\{D_1,D_2,D_3\}$ and $l\ge\sqrt 3$.

\end{exam}

\section{Complements and further results}\label{s6}

In this Section we present two generalizations of Theorem \ref{th}. The first one deals with a more general class of constraints $D_1,\dots,D_k$ while in the second one we consider a larger class of admissible sets.\\

\begin{cor}
Let $D_1,\dots,D_k$ be $k$ disjoint compact sets in $\R^d$ and let $l\ge St(d_1,\dots,d_k)$, i.e. such that there exists a closed connected set $C$ of length $\H^1(C)=l$, which intersects all the sets $D_1,\dots,D_k$. Then the optimization problem
\be\label{enimcpt}
\min\left\{\E(\Gamma;\V):\ \Gamma\in CMG,\ l(\Gamma)=l, \V\subset V(\Gamma),\ \Gamma\in Adm(\V;D_1,\dots,D_k)\right\}
\ee 
admits a solution, where we say that $\Gamma\in Adm(\V;D_1,\dots,D_k)$, if there exists an immersion $\gamma:\Gamma\to\R^d$ such that for each $j=1,\dots,k$ there is $V_j\in\V$ such that $\g(V_j)\in D_j$.
\end{cor}

\begin{proof}
As in Theorem \ref{th2}, we can restrict our attention to the connected metric trees $\Gamma$ with the same vertices $V(\Gamma)=\{V_1,\dots,V_N\}$ and edges $E(\Gamma)=\{e_{ij}\}_{ij}$. Moreover, we can suppose that $\V=\{V_1,\dots,V_k\}$ is fixed. By the compactness of the sets $D_j$, we can take a minimizing sequence $\Gamma_n$ and immersions $\gamma_n$ such that for each $j=1,\dots,k$, we have $\gamma_n(V_j)\to X_j\in D_j$, as $n\to\infty$. The claim follows by the same argument as in Theorem \ref{th}.
\end{proof}

Theorem \ref{th} can be restated in the more general framework of the metric spaces of finite Hausdorff measure, which is the natural extension of the class of the one dimensional subspaces of $\R^d$ of finite length. In fact, for any compact connected metric space (shortly CCMS) $(C,d)$, we consider the one dimensional Hausdorff measure $\H^1_d$ with respect to the metric $d$ and the Sobolev space $H^1(C)$ obtained by the closure of the Lipschitz functions on $C$, with respect to the norm $\|u\|_{H^1(C)}^2=\|u\|^2_{L^2(\H^1_d)}+\|u'\|^2_{L^2(\H^1_d)}$, where $u'$ is defined as in the case $C\subset\R^d$. The energy $\E(C;\V)$ with respect to the set $\V\subset C$ is defined as in \eqref{energy1}. As in the case of metric graphs, we define an immersion $\g:C\to\R^d$ as a continuous map such that for any arc-length parametrized curve $\eta:(-\eps,\eps)\to C$, we have that $|(\g\circ\eta)'(t)|=1$ for almost every $t\in(-\eps,\eps)$. As a consequence of Theorem \ref{th}, we have the following:

\begin{cor}
Consider the set of points $\D=\{D_1,\dots,D_k\}\subset\R^d$ and a positive real number $l\ge St(D_1,\dots,D_k)$. Then the following optimization problem has solution:
\be\label{enimC}
\min\left\{\E(C;\V):\ (C,d)\in CCMS,\ \H^1_d(C)\le l,\ C\in Adm(\V;D_1,\dots,D_k)\right\},
\ee 
where the admissible set $Adm(\V;\{D_1,\dots,D_k\})$ is the set of connected metric spaces, for which there exists an immersion $\gamma:\Gamma\to\R^d$ such that $\g(\V)=\{D_1,\dots,D_k\}$. Moreover, the solution of the problem \eqref{enimC} is a connected metric graph, which is a tree of at most $2k$ vertices and $2k-1$ edges. 
\end{cor}

\begin{proof}
Repeating the construction from Theorem \ref{th1}, we can restrict our attention to the class of metric graphs. The thesis follows from Theorem \ref{th}.
\end{proof}

The results from Theorem \ref{th1} and Theorem \ref{th}, hold also for other cost functionals as, for example, the first eigenvalue of the Dirichlet Laplacian:
\begin{equation}
\lambda_1(\Gamma;\mathcal V)=\min\left\{\int_\Gamma|u'|^2\,dx:\ u\in H^1_0(\Gamma),\ \int_\Gamma u^2\,dx=1\right\},
\end{equation} 
where $\Gamma$ is a metric graph and $\mathcal{V}\subset V(\Gamma)$ is a set of vertices, where a Dirichlet boundary conditions are imposed. Reasoning as in Remark \ref{ps}, we have that among all connected metric graphs (shortly, CMG) of fixed length $l$ and  with at least one Dirichlet vertex, the one with the lowest first eigenvalue is given by the segment $[0,l]$, with Dirichlet condition in $0$. Moreover, for any pair $D_1,D_2\in\R^d$ and any $l\ge |D_1-D_2|=:l-\epsilon$ the solution of
\begin{equation}\label{lb1}
\min\left\{\lambda_1(\Gamma;\mathcal V):\ \Gamma\in CMG, l(\Gamma)=l, \mathcal V\subset V(\Gamma),\exists \gamma:\Gamma\to\R^d\ \hbox{immersion},\ \gamma(\mathcal V)=\mathcal D\right\},
\end{equation}
is the graph described in Figure \ref{f1ex2}, i.e. the solution of \eqref{enim2} from Example \ref{example2}. In the case when the set $\mathcal{D}$ is given by three points disposed in the vertices of an equilateral triangle, the solutions of \eqref{lb1} are quantitatively the same (see Figure \ref{f2ex3}) as the solutions of \eqref{enim3} from Example \ref{example3}. In general, we have the following existence result

\begin{teo}
Consider a set of distinct points $\D=\{D_1,\dots,D_k\}\subset\R^d$ and a positive real number $l\ge St(\D)$. Then there exists a connected metric graph $\Gamma$, a set of vertices $\V\subset V(\Gamma)$ and an immersion $\gamma:\Gamma\to\R^d$ which are solution of the problem \eqref{lb1}. Moreover, $\Gamma$ can be chosen to be a tree of at most $2k$ vertices and $2k-1$ edges.
\end{teo}
\begin{proof}
The proof is identical to the one of Theorem \ref{th}.
\end{proof}

\begin{oss}
The question of existence of an optimal graph is open for general cost functionals $J$ spectral type, i.e. $J=F(\lambda_1,\lambda_2,\dots,\lambda_k,\dots)$, where $F:\R^\N\to\R$ is a real function and $\lambda_k$ is the $k$-th eigenvalue of the Dirichlet Laplacian:
\begin{equation}
\lambda_k(\Gamma;\mathcal V)=\min_{K\subset H^1_0(\Gamma)}\max\left\{\int_\Gamma|u'|^2\,dx:\ u\in K,\ \int_\Gamma u^2\,dx=1\right\},
\end{equation} 
where the minimum is over all $k$ dimensional subspaces $K$ of $H^1_0(\Gamma)$. In fact, the crucial point in the proof of Theorem \ref{th} is the reduction to the class of connected metric trees with number of vertices bounded by some universal constant. This reduction becomes a rather involved question even for the simplest spectral functionals $\lambda_k$ for $k\ge2$.

A similar difficulty occurs for other kinds of shape optimization problems for graphs, like for instance the optimization of integral functionals
$$J(C)=\int_C j(x,w_C)\,d\H^1$$
being $w_C$ the solution of
$$\min\left\{\int_C\Big(\frac12|u'|^2-fu\Big)\,d\H^1\ :\ u\in H^1_0(C,\D)\right\},$$
where $f$ is a continuous function on $\R^d$.
\end{oss}

\begin{ack}
The authors would like to thank Dorin Bucur for some useful suggestions during the preparation of the work. They are also grateful to Mihail Minchev for the discussions on the metric graphs and explaining them the physical point of view on the topic.   
\end{ack}


\bigskip\noindent
Giuseppe Buttazzo:
Dipartimento di Matematica,
Universit\`a di Pisa\\
Largo B. Pontecorvo 5,
56127 Pisa - ITALY\\
{\tt buttazzo@dm.unipi.it}\\
{\tt http://www.dm.unipi.it/pages/buttazzo/}

\bigskip\noindent
Berardo Ruffini:
Scuola Normale Superiore di Pisa,\\
Piazza dei Cavalieri 7, 
56126 Pisa -ITALY\\
{\tt berardo.ruffini@sns.it}

\bigskip\noindent
Bozhidar Velichkov:
Scuola Normale Superiore di Pisa\\
Piazza dei Cavalieri 7, 
56126 Pisa - ITALY\\
{\tt b.velichkov@sns.it}

\end{document}